\documentclass[12pt]{amsart}
\usepackage{amssymb}
\usepackage{amsmath}
\usepackage{amsthm}
\newtheorem{theorem}{Theorem}[section]
\newtheorem{lemma}[theorem]{Lemma}
\newtheorem{remark}[theorem]{Remark}

\newtheorem{corollary}[theorem]{Corollary}
\newtheorem{proposition}[theorem]{Proposition}

\newtheorem{lem-def}[theorem]{Lemma-Definition}
\DeclareRobustCommand\longtwoheadrightarrow
{\relbar\joinrel\twoheadrightarrow}

\newcommand{\hooklongrightarrow}{\lhook\joinrel\longrightarrow}
\newcommand{\I}{\mathbb I}
\newcommand{\R}{\mathbb R}
\renewcommand{\S}{\mathbb S}

\newcommand{\N}{\mathbb N}
\newcommand{\Z}{\mathbb Z}
\newcommand{\Q}{\mathbb Q}

\newcommand{\T}{\mathbb{T}}
\newcommand{\U}{\mathbb{U}}

\renewcommand{\H}{\mathbb{H}}
\renewcommand{\tt}{\mathcal{T}}
\def\op{\operatorname}

\def\al{\alpha}
\def\as#1{\renewcommand\arraystretch{#1}}
\def\bs{\vskip.5cm}
\def\be{\beta}

\def\com{{\op{com}}}
\def\cv{\op{Cvx}}

\def\d{\Delta}
\def\dta{\delta}

\def\e{\medskip}
\def\ep{\epsilon}
\def\eqr{\op{EqRk}}
\def\eqrat{\op{EqRk}^{\op{rat}}}
\def\eqri{\op{EqRk}^{\op{irrat}}}

\def\g{\Gamma}
\def\ga{\gamma}

\def\gen#1{\big\langle\, {#1} \,\big\rangle}

\def\gga{\gen{\g,\ga}}
\def\ggb{\gen{\g,\be}}
\def\ggd{\gen{\g,\dta}}
\def\gge{\gen{\g,\ep}}

\def\gm{\g_\mu}
\def\gn{\g_\nu}

\def\gq{\g_\Q}

\def\grr{\g_\R}
\def\gsme{\g_{\op{sme}}}

\def\hg{H(\ga)}
\def\hra{\hooklongrightarrow}

\def\hq{\mathbb{H}(\gq)}
\def\hk{\hookrightarrow}
\def\imp{\ \Longrightarrow\ }

\def\incom{\op{incom}}

\def\incrg{\op{IncRk}(\g)}
\def\inf{\op{inf}}

\def\inii{\op{Init}(I)}
\def\init{\op{init}}
\def\iso{\lower.3ex\hbox{\as{.08}$\begin{array}{c}\lra\\\mbox{\tiny $\sim\,$}\end{array}$}}
\def\ism{\lower.3ex\hbox{\as{.08}$\begin{array}{c}\,\to\\\mbox{\tiny $\sim\,$}\end{array}$}}
\def\k{\op{Ker}}
\def\kb{\overline{K}}
\def\km{k_\mu}

\def\kx{K[x]}

\def\La{\Lambda}
\def\lc{\op{lc}}
\def\lg{l\raise.6ex\hbox to.2em{\hss.\hss}l}

\def\lra{\,\longrightarrow\,}
\def\lx{\operatorname{lex}}

\def\mb{\bar{\mu}}

\def\minf{\om_{-\infty}}

\def\om{\omega}

\def\opp{^{\op{opp}}}
\def\ord{\op{ord}}
\def\p{\mathfrak{p}}

\def\pcv{\op{Prin}}
\def\pr{\op{prk}}

\def\rk{\op{rk}}
\def\rii{\R^\I_{\lx}}
\def\rlex{\R^I_{\lx}}

\def\rrk{\op{rr}}

\def\sii{\ \Longleftrightarrow\ }
\def\sme{{\mbox{\tiny $\op{sme}$}}}

\def\spv{\op{Spv}}
\def\ss{\mathcal{L}}
\def\sup{\op{sup}}
\def\supp{\op{supp}}

\def\ttt{\mathcal{T}}
     
\def\vb{\bar{v}}

\title{Small extensions of abelian ordered groups}

\makeatletter
\@namedef{subjclassname@2010}{%
  \textup{2010} Mathematics Subject Classification}

\thanks{Partially supported by grant MTM2016-75980-P from the Spanish MEC}

\author[Nart]{Enric Nart}
\address{Departament de Matem\`{a}tiques,         Universitat Aut\`{o}noma de Barcelona,         Edifici C, E-08193 Bellaterra, Barcelona, Catalonia}
\email{nart@mat.uab.cat}
\date{}

\begin{document}

\begin{abstract}
Let $\g$ be a totally ordered abelian group. We use Hahn's embedding theorem to construct a totally ordered set $\g\hk \gsme$ which classifies small extensions of $\g$.  This small-extensions closure $\gsme$ is complete and plays a crucial role in the description of equivalence classes of valuations on the polynomial ring $\kx$ over a field $K$. 
\end{abstract}

\maketitle

\section*{Introduction}

Let $\g$ be a totally ordered abelian group and let $\gq=\g\otimes_\Z\Q$ be its divisible hull.

Let  $\g\hk\La$  be an order-preserving embedding of ordered groups and let $\g^\com\subset\La$ be the subgroup of all elements in $\La$ which are commensurable over $\g$:
$$
\g^\com=\left\{\ga\in\La\mid n\ga\in\g \ \mbox{for some }n\in\Z_{>0}\right\}\subset\La.
$$

We say that $\g\hk\La$ is a \emph{small extension} if $\La/\g^\com$ is a cyclic group.

In this paper, we use Hahn's embedding theorem to construct certain universal totally ordered sets
$$
\g\subset \gq\subset\grr\subset\gsme
$$
which classify small extensions.

More precisely, for a small extension $\g\hk\La$ as above, let $\g^\com\ism\Delta\subset \gq$ be the canonical embedding into the divisible hull. Then, for any choice of a generator $\ga\in\La$ of the cyclic quotient $\La/\g^\com$, there exists a unique element $\be\in\gsme$ and a unique isomorphism $\La\ism\gen{\Delta,\be}$ which sends $\ga$ to $\be$ and acts as the canonical isomorphism on $\g^\com$. Moreover, the small extension $\g\hk\La$ preserves the rank if and only if $\be$ belongs to $\grr$. 

The content of the paper is as follows.
In section 1, we review some basic facts on abelian orderd groups, including Hahn's embedding theorem.
In section 2, we construct $\gsme$, the small-extensions closure of $\g$. In section 3, we prove some topological properties of this object: it inherits a canonical total order and $\gq$ is dense in $\gsme$. However, the most relevant property of $\gsme$ is its completeness in a strong sense: every non-empty subset admits an infimum and a supreme.

Finally, section 4 is devoted to give some hints about the applications of these constructions to valuation theory.

%The motivation for this construction lies on its applications to valuation theory.

A valuation on  a commutative ring $A$ is a mapping
$$
\mu\colon A\lra \La\infty
$$
where $\La$ is an ordered group, satisfying the following conditions:\e

(0) \ $\mu(1)=0$, \ $\mu(0)=\infty$,\e

(1) \ $\mu(ab)=\mu(a)+\mu(b),\quad\forall\,a,b\in A$,\e

(2) \ $\mu(a+b)\ge\min\{\mu(a),\mu(b)\},\quad\forall\,a,b\in A$.\e

The \emph{support} of $\mu$ is the prime ideal $\p=\mu^{-1}(\infty)\in\op{Spec}(A)$. 
The valuation $\mu$ induces a valuation on the field $\kappa(\p)=\op{Frac}(A/\p)$. 

The \emph{value group} of $\mu$ is the subgroup $\gm\subset \La$ generated by $\mu\left(A\setminus\p\right)$.

%Denote the maximal ideal, valuation ring and residue class field of this valuation on $\kappa(\p)$, by $$\m_\mu\subset \oo_\mu\subset \kappa(\p),\qquad k_{\mu}=\oo_\mu/\m_\mu.$$
%Note that $\kappa(0)=K(x)$, while for $\p\ne0$ the field $\kappa(\p)$ is a simple finite extension of $K$. 

%Thus, a valuation $\mu$ on $\kx$ determines a valuation on a simple field extension $L/K$, either algebraic or transcendental.\bs

Two valuations $\mu$, $\nu$ on $A$ are \emph{equivalent} if there is an isomorphism of ordered groups $\varphi\colon \gm \ism\gn$ fitting into a commutative diagram
$$%\begin{equation}\label{equivalence}
\as{1.4}
\begin{array}{ccc}
\gm\infty&\stackrel{\varphi}\lra\ &\!\!\gn\infty\\
\quad\ \mbox{\scriptsize$\mu$}&\nwarrow\ \nearrow&\!\!\!\mbox{\scriptsize$\nu$}\quad\\
&A&
\end{array}
$$%\end{equation}

In this case, we write $\mu\sim\nu$. 
The \emph{valuative spectrum} of $A$ is the set $\spv(A)$ of equivalence classes of valuations on $A$. We denote by $[\mu]\in\spv(A)$ the equivalence class of $\mu$.

Any ring homomorphism $A\to B$ induces a restriction mapping $\spv(B)\to\spv(A)$, %determined by the assignment:$$B\stackrel{\mu}\lra\La\infty\quad\longmapsto\quad A\lra B\stackrel{\mu}\lra\La\infty,$$
which behaves well on equivalence classes.
%\subsection*{Extensions of a valuation on $K$}

For any field $K$ we may consider the relative affine line $\spv(\kx)\to\spv(K)$.

Given any valuation $v$ on $K$, the fiber $\ttt_v$ of the equivalence class $[v]\in\spv(K)$ is called the \emph{valuative tree} over the valued field $(K,v)$.

$$
\as{1.4}
\begin{array}{ccc}
\ttt_v&\hra &\spv(\kx)\\
\downarrow&&\downarrow\\
\mbox{$[v]$}&\hra &\spv(K)
\end{array}
$$

Let $\g=v(K^*)$ be the value group of $v$. 
An element $[\mu]\in\ttt_v$ is an equivalence class of valuations $\mu$ on $\kx$ whose restrictions to $K$ are equivalent to $v$. In other words, there exists an embedding of ordered groups $\iota\colon\g\hk\gm$, fitting into a commutative diagram 
$$
\as{1.4}
\begin{array}{ccc}
\kx&\stackrel{\mu}\lra&\gm\infty\\
\uparrow&&\ \uparrow\mbox{\tiny$\iota$}\\
K&\stackrel{v}\lra&\g\infty
\end{array}
$$
The extension $\iota\colon\g \hk\gm$ is always a small extension.

In the case $\rk(\g)=1$ and $K$ algebraically closed, the tree $\ttt_v$ has been extensively studied. It admits a structure of a Berkovich space and has relevant analytical properties \cite{Vtree,Gja,Bch}.

There is on-going research on the problem of extending these properties to arbitrary valued fields $(K,v)$. In this regard, the small-extensions closure $\gsme$ of $\g$  plays a crucial role. 

In section 4, we point out two modest steps in this direction. We show that $\gsme$ parameterizes certain paths in $\ttt_v$. Also, thanks to the completeness of $\gsme$, we see that valuations on $\kx$ admit a topological interpretation completely analogous to that introduced by Berkovich in the rank-one case.

\section{Background on abelian ordered groups}

\subsection{Ordered sets and groups}

Throughout the paper, an \emph{ordered set} will be a set equipped with a total ordering. We agree that $0\not\in\N$.\e

\noindent{\bf Notation. }Let $I,\,J$ be ordered sets.

\begin{itemize}
\item $I\infty$ is the ordered set obtained by adding a (new) maximal element, which is formally denoted as $\infty$. 
\item $I\opp$ is the ordered set obtained by reversing the ordering of $I$.
\item For $S,T\subset I$ and $i\in I$, the following expressions have the obvious meaning
$$i<S,\qquad i>S,\qquad S<T.$$
\item For all $i\in I$, we denote $I_{< i}=\{j\in I\mid j< i\}\subset I_{\le i}=\{j\in I\mid j\le i\}$. 
\item $I+J$ is the disjoint union $I\sqcup J$ with the total ordering which respects the orderings of $I$ and $J$ and satisfies $I<J$.
\end{itemize}\e

An \emph{initial segment} of $I$ is a subset $S\subset I$ such that
$$
i\in S\imp I_{\le i}\subset S.
$$
We denote by $\inii$ the set of initial segments of $I$.
Clearly, $\inii$ is an ordered set with respect to inclusion.\e

A mapping $\iota\colon I\to J$  is an \emph{embedding} if it strictly preserves the order.
%$$i<j\imp\iota(i)<\iota(j),\quad\forall\,i,j\in I.$$
We also say that $\iota\colon I\to J$ is an \emph{extension} of $I$. 

An \emph{isomorphism} of ordered sets is an onto embedding.
The \emph{order-type} of an ordered set is the class of this set up to isomorphism. \e

%For well-ordered sets the order-type is an ordinal number, like $$\tp(I)=7, \quad\omega, \quad \omega^3\cdot 4+11,$$where $\omega$ is the order-type of $\N$. We agree that $0\not\in\N$.%\e

%For a set which is not well-ordered, the order-type is not a ``number" of any kind. \bs%However, we say that $$\tp(I)\le \tp(J)$$if $I$ may be embedded into $J$ as an ordered set.\bs

%\subsection{Ordered groups}

An \emph{ordered group} $(\g,\le)$ is an (additive) abelian group $\g$ equipped with a total ordering $\le$, which is compatible with the group structure.
%$$\be<\ga\imp \be+\rho<\ga+\rho,\qquad\forall\,\be,\ga,\rho\in\g.$$

For all $\ga\in\g$, we denote $|\ga|=\max(\ga,-\ga)$.

An ordered group $\g$ has no torsion. In fact, any non-zero $\ga\in\g$  satisfies $$n|\ga|>|\ga|>0,\quad \forall\,n\in\N.$$

An \emph{embedding/extension/isomorphism} of ordered groups is a group homomorphism which is simultaneously an embedding/extension/isomorphism  of ordered sets.%\bs%\newpage

A basic example of ordered group is $\R^n_{\lx}$, the additive group $(\R^n,+)$ equipped with the lexicographical order. 

Also, any subgroup of an ordered group inherits the structure of an ordered group.

\subsection*{Hahn sum and Hahn product} Let $I$ be an ordered set, and let $(\g_i)_{i\in I}$ be a  family of ordered groups parameterized by $I$. 

Their \emph{Hahn sum} is the direct sum equipped with the lexicographical order.
$$\coprod\nolimits_{i\in I}\g_i:=\bigoplus\nolimits_{i\in I}\g_i.$$

For any element $a=(a_i)_{i\in I}$ in the product  $\prod_{i\in I}\g_i$, the \emph{support} of $a$ is the subset 
$$\op{supp}(a)=\left\{i\in I\mid a_i\ne0\right\}\subset I.$$

We define the \emph{Hahn product} $$\left(\prod\nolimits_{i\in I}\g_i\right)_{\!\lx}\subset\prod\nolimits_{i\in I}\g_i$$ as the subgroup formed by all elements whose support is a well-ordered subset of $I$, with respect to the ordering induced by that of $I$.

It is easy to check that it makes sense to consider the lexicographical order on this subgroup.\e

Clearly, the Hahn product is an extension of the Hahn sum:
$$\coprod\nolimits_{i\in I}\g_i\subset \left(\prod\nolimits_{i\in I}\g_i\right)_{\!\lx}.$$

If $\g_i=\g$ for all $i\in I$, then we use the notation
$$
\g^{(I)}\subset \g^I_{\lx},
$$
for the Hahn sum and product, respectively.

\subsection*{Divisible hull}

%$$\be<\ga\imp\iota(\be)<\iota(\ga).$$Also, we say that $\d$ is an \emph{extension} of $\g$. An onto embedding is called an \emph{isomorphism} of ordered groups.\e 

%A torsion-free group $G$ is \emph{divisible} if for all $\ga\in G$ and all $n\in\N$, there exists (a necessarily unique) $\beta \in G$ such that $n\beta=\ga$.\e

For any ordered group $\g$, its \emph{divisible hull} is the group $$\gq:=\g\otimes\Q.$$ This group $\gq$ has a natural structure of ordered group with the ordering determined by the condition
$$
\ga\otimes(1/n)<\be\otimes (1/m)\sii m\ga<n\be,
$$
for all $n,m\in\N$ and all $\ga,\be\in\g$.

Since $\g$ has no torsion, it may be embedded in a unique way into $\gq$ as an ordered group. 
The divisible hull of $\g$ is the minimal divisible extension of $\g$.

\begin{lemma}\label{MInDiv}
For any embedding $\iota\colon\g\hk \La$ into a divisible ordered group $\La$, there exists a unique embedding of $\gq$ into $\La$ such that $\iota$ coincides with the composition $\g\hk \gq\hk\La$. \e
\end{lemma}

\subsection{Convex subgroups and rank}\label{secCvx}

Let us fix an ordered group $\g$.

Given a subgroup $H\subset \g$, the quotient $\g/H$ inherits a structure of ordered group if and only if $H$ is a \emph{convex} subgroup; that is, 
$$
\rho\in\g,\  \;\ga\in H,\ \;|\rho|\le|\ga|\ \imp \ \rho\in H.
$$
In this case, we may define an ordering in $\g/H$ by:
$$
\al+H<\be+H \sii \al+H\ne \be+H \ \mbox{ and }\ \al<\be.
$$

The notation $\al+H<\be+H$ is compatible with the natural meaning of such an inequality for arbitrary subsets of $\g$. %That is, every element in $\al+H$ is smaller than any element in $\be+H$.

%In particular, if $H$ is a convex subgroup, the quotient homomorphism $\g\to\g/H$ preserves the ordering. That is,$$\al\le\be\ \imp\ \al+H\le\be+H.$$Only convex subgroups have this property.

%The following result shows that only convex subgroups have this property. 

\begin{lemma}
Let $f\colon \g\to\d$ be an order-preserving group homomorphim between two ordered groups. Then, $\k(f)$ is a convex subgroup of $\g$ and the natural isomorphism between $\g/\k(f)$ and $f(\g)$ is order-preserving too.    
\end{lemma}

\begin{lemma}
The convex subgroups of $\g$ are totally ordered by inclusion.
\end{lemma}

\begin{proof}
Let $H$, $H'$ be convex subgroups such that there exists $\ga\in H\setminus H'$.

Then, for all $\be\in H'$ we must have $|\be|<|\ga|$, so that $H'\subset H$. 
\end{proof}

\noindent{\bf Definition. }Let $\cv=\cv(\g)$ be the ordered set of all \emph{proper} convex subgroups, ordered by increasing inclusion
$$
\{0\}\subset \cdots \subset H\subset \cdots  \subsetneq \g
$$ 
The order-type of $\cv$ is called the $\emph{rank}$ of $\g$, and is denoted $\rk(\g)$.

We may identify $\cv\!\infty$ with the ordered set of all convex subgroups of $\g$, by letting $\infty$ represent the whole group $\g$. \e

\noindent{\bf Examples. }
\begin{itemize}
\item $\rk(\Z)=\rk(\Q)=\rk(\R)=1$. \e

\item $\rk(\R^n_{\lx})=n$. The sequence of convex subgroups is
$$
\{0_{\R^n}\}\ \subset \ \cdots\ \subset \{0\}^k\times\R^{n-k}\ \subset\ \cdots\ \subset \ \R^n_{\lx},
$$
%\item $\rk\left(\R^\N_{\lx}\right)=\tp\left((\N\infty)\opp\right)$, \quad $\rk\left(\R^\Q_{\lx}\right)>\tp\left((\Q\infty)\opp\right)$.\e 
\item $\rk(\g)=\rk(\gq)$.  
%$$\{0\}\subset H_1\subset\cdots\subset H_{n-1}\subset \R^n_{\lx},$$where $H_i=\{(a_j)_{1\le j\le n}\mid a_j=0 \mbox{ for all }j\le n-i\}$.
%\item If $I$ is a well-ordered set, then $$\rk\left(\R^{(I)}\right)=\rk\left(\R^I_{\lx}\right)=\op{order-type}(I),$$ but for an arbitrary $I$ the rank of the Hahn sum and product may be larger than the order-type of $I$.
\end{itemize}

\subsection*{Principal convex subgroups}
For any $\ga\in\g$, we denote by $\hg$ the convex subgroup of $\g$ generated by $\ga$. That is,
$$
\hg=\left\{\be\in \g\mid |\be|\le n|\ga| \mbox{ for some } n\in\N \right\}.
$$ 

Equivalently, $\hg$ is the intersection of all convex subgroups containing $\ga$.

These convex subgroups $\hg$ are said to be \emph{principal}. \e

\noindent{\bf Definition. }Let $I=\pcv(\g)$ be the ordered set of \emph{non-zero} convex principal subgroups of $\g$, ordered by decreasing inclusion.
%$$\cdots \supset H(\ga)\supset \cdots  $$ 

The order-type of $I$ is called the \emph{principal rank} of $\g$, and is denoted $\pr(\g)$.
\e

We may identify $I\infty$ with a set of indices parameterizing all principal convex subgroups of $\g$. For any $i\in I$ we denote by $H_i$ the corresponding  principal convex subgroup. We agree that $H_\infty=\{0\}$.

Then, according to our convention, for any pair of indices $i,j\in I\infty$, we have
$$
i<j \sii H_i\supsetneq H_j.
$$

\begin{lemma}\label{allH}
Every convex subgroup $H\subset \g$ satisfies $H=\bigcup_{i\in I,\,H_i\subset H}H_i$.
\end{lemma}

\begin{proof}
For all $\ga\in H$, the principal convex subgroup $H(\ga)$ is contained in $H$.
\end{proof}

\begin{corollary}
If $I$ is well-ordered, then all convex subgroups are principal. 
\end{corollary}

\begin{proof}
For any convex subgroup  $H$, the subset $\{i\in I\mid H_i\subset H\}\subset I$ has a minimal element $i_0$. By Lemma \ref{allH}, $H=H_{i_0}$.
\end{proof}

\subsection*{Skeleton of an ordered group}% and immediate extensions}

If $H=\hg\in I$, we denote
by $H^*$ the union of all principal convex subgroups not containing $\ga$. 

Clearly, $H^*\subsetneq H$ is a convex subgroup (not necessarily principal) which
is the immediate precedessor of $H$ in the ordered set $\cv\!\infty$. 

In particular, \emph{the quotient $H/H^*$ is an ordered group of rank one}.\e

\noindent{\bf Definition. }This quotient $H/H^*$ is said to be the \emph{component} of $\g$ determined by the non-zero principal convex subgroup $H$. 

The component of $\g$ determined by any $i\in I$ will be denoted as 
$$
C_i=C_i(\g)=H_i/H_i^*.
$$

The \emph{skeleton} of $\g$ is the pair $\left(I,(C_i)_{i\in I}\right)$.

\begin{lemma}\label{rkEmbed}
Let $\g\hk\La$ be an extension of ordered groups. The following mappings are embeddings of ordered sets:
$$
%\as{1.4}
\begin{array}{ccl}
\cv(\g)\lra\cv(\La),&\quad H&\!\longmapsto \;H_\La\mbox{ \ convex subgroup  generated by }H,\\
\pcv(\g)\stackrel{\iota}\lra\pcv(\La),&\quad H(\ga)&\!\longmapsto \;H_\La(\ga)\mbox{ \ convex subgroup generated by }\ga,
\end{array}
$$  
%Thus, $\rk(\g)\le \rk(\La)$ and $\ \pr(\g)\le\pr(\La)$.

Moreover, if $i\in I=\pcv(\g)$ is the index that corresponds to $H(\ga)$, then the embedding $H(\ga)\subset H_\La(\ga)$ induces an embedding $C_i(\g)\hk C_{\iota(i)}(\La)$. %between their respective components. 
\end{lemma}

\noindent{\bf Definition. }The extension $\g\hk\La$ is \emph{immediate} if it preserves the skeleton. That is, it induces an isomorphism $\pcv(\g)\simeq \pcv(\La)$ of ordered sets, and isomorphisms $C_i(\g)\simeq C_{\iota(i)}(\La)$ between all the components.

%If an extension $\g\hk\La$ is immediate, then $\pr(\g)=\pr(\La)$ by the very definition of the principal rank. Also, one has $\rk(\g)=\rk(\La)$ by Lemma \ref{IConv}.

%The converse implication is not true. All non-trivial subgroups of $\R$ have rank one, but they have many different skeletons.

\subsection*{Ordered groups with a prefixed skeleton}

Let $I$ be an ordered set and $(C_i)_{i\in I}$ a family of ordered groups of rank one, parameterized by $I$.

The Hahn sum and product
$\ \coprod_{i\in I}C_i\subset \left(\prod_{i\in I}C_i\right)_{\!\lx}$ \
have both skeleton $\left(I,(C_i)_{i\in I}\right)$.

More precisely, let $\g$ denote any one of these two groups. For each $i\in I$, consider the following subgroup of $\g$: 
$$
H_i=\{(a_j)_{j\in I}\mid a_j=0\mbox{ for all }j<i\}.
$$
Then, $H_i$ is the principal subgroup of $\g$ generated by any $(a_j)_{j\in I}\in H_i$ with $a_i\ne0$. 

Also, the assignment $i\mapsto H_i$ determines an isomorphism of ordered sets between 
$I$ and $\pcv(\g)$. In particular, $\pr(\g)$ is the order-type of $I$. 

Moreover, the projection 
$
H_i\to C_i$, sending $ (a_j)_{j\in I}\mapsto a_i$,
induces an isomorphism of ordered groups between $H_i/H_i^*$ and $C_i$.

\subsection*{Relationship between rank and principal rank}

%Since $\pr(\g)=\tp(I)$ is related to the set of components of $\g$, it gives a more precise idea about the structure of $\g$ as an ordered group than  $\rk(\g)=\tp(\cv)$.

%However, the ordered sets $I$ and $\cv$ determine one to each other.\e

The ordered sets $I$ and $\cv\infty$ determine one to each other.

\begin{lemma}\label{ConvI}
The set $I$ is the subset of $\cv\!\infty$ formed by all elements admitting an immediate predecessor.%\footnote{Recall that the ordering of $I$ is the opposite of the ordering induced by $\cv\!\infty$.}
\end{lemma}

\begin{proof}
Any non-zero principal convex subgroup $H$ has an immediate predecessor $H^*$.

Conversely, if $H'\subsetneq H$ is an immediate predecessor of a convex subgroup $H$, then $H$ is the principal subgroup generated by any $\ga\in H\setminus H'$.
\end{proof}

To any initial segment $S\in\inii$, we may associate the convex subgroup
$$
H_S=\bigcup\nolimits_{i\in I,\,i>S}H_i.
$$

\begin{lemma}\label{IConv}
The assignment $S\mapsto H_S$ determines an isomorphism of ordered sets:
$$
\inii\opp\lra \cv\infty.
$$

The  inverse isomorphism assigns $\ H\mapsto S_H:=\{i\in I\mid H_i\supsetneq H\}$.
\end{lemma}

\begin{proof}
The mapping $S\mapsto H_S$ is an embedding of ordered sets. Indeed, if $T\supsetneq S$ and 
$i\in T\setminus S$, then $H_T\subset H^*_i\subsetneq H_i\subset H_S$.

Also, it is an onto map because $H=H_{S_H}$ by Lemma \ref{allH}.
\end{proof}

The following table illustrates how $\cv$ is constructed from $I$. 

%Any $i\in I$ determines two ``trivial" initial segments: $I_{\le i}\supsetneq I_{< i}$.There are two more trivial initial segments, the whole set $I$ and the empty subset $\emptyset$.

An initial segment $\emptyset\subsetneq S\subsetneq I$ is non-trivial if neither $S$ has a maximal element, nor $I\setminus S$ has a minimal element.

%To each of these initial segments the above correspondence assigns the following convex subgroups.\e\textbf{}

\begin{center}
\as{1.1}
\begin{tabular}{|c|c|}
\hline
{\bf initial segment}&{\bf convex subgroup} \\\hline
$\emptyset$&$\g$\\
\hline
$I$&$\{0\}$\\\hline
$I_{\le i}$&$H_i^*$\\\hline
$I_{< i}$&$H_i$\\\hline
$S$ non-trivial&$H_S$ non-principal\\\hline
\end{tabular}
\end{center}\bs

%The subgroup $H_i^*$ is principal if and only if $i$ has an immediate successor (say $i+1$) in $I$. In this case, $I_{\le i}=I_{< i+1}$ and $H_i^*=H_{i+1}$.

The non-principal convex subgroups arise in a two-fold way. Either from non-trivial initial segments, or from segments of the form $I_{\le i}$ for $i\in I$ having no immediate successor in $I$, in which case $H_i^*$ is non-principal.%\footnote{That is, principal convex subgroups $H_i$ having no immediate  principal convex subgroup predecessor $ H_i\supsetneq H_{i+1}$.}

For instance, suppose that $I=\Q$. Then, every non-trivial initial segment of $I$ determines a non-principal convex subgroup parameterized by a real number. On the other hand, every rational number $q$ determines two convex subroups $H_q^*\subsetneq H_q$, from which only $H_q$ is principal.

%Thus, $\cv=\{0\}+(\R\opp)^\circ$, where $(\R\opp)^\circ$ is a real line with the opposite order, in which every rational number has been doubled by adding an immediate predecessor of it.

\begin{corollary}\label{Gpp}
The following conditions are equivalent.
\begin{enumerate}
\item $\g$ is a principal convex subgroup.
\item $I$ has a minimal element.
\item $\cv$ has a maximal element (immediate predecessor of $\g$).
\end{enumerate}
\end{corollary}

\noindent{\bf Example. }The Hahn product
$\R^\N_{\lx}=\R^\N$ is a principal convex subgroup (of itself), generated by any $(a_i)_{i\in\N}$ with $a_1\ne0$. 
On the other hand, the Hahn product
$$\R^{(\N\opp)}=\R^{\N\opp}_{\lx}\subsetneq\R^{\N\opp}$$
is not a principal convex subgroup of itself.

\subsection{Arquimedean classes  and Hahn's theorem}\label{secHahn}
%Let $\g$ be an ordered group. 
Two non-zero elements $\be,\ga\in\g$ are \emph{arquimedeanly equivalent} if there exist $n,m\in\N$ such that
$$
|\be|<n|\ga| \quad \mbox{ and }\quad |\ga|<m|\be|.
$$ 
In this case, we write $\be\sim\ga$. 

This defines an equivalence relation on $\g\setminus\{0\}$. The equivalence classes are in canonical bijection with the set $I=\pcv(\g)$ of non-zero principal convex subgroups.

\begin{lemma}\label{IArq}
 Two non-zero elements $\be,\ga\in\g$ are arquimedeanly equivalent if and only if they generate the same convex subgroup: $H(\be)=H(\ga)$. 
\end{lemma}

We say that $\g$ is \emph{arquimedean} if all non-zero elements are arquimedeanly equivalent.   

\begin{proposition}\label{rk1}%{\cite[Prop. 2.1.1]{valuedfield}}
Let $\g$ be a non-trivial ordered group. The following conditions are equivalent.
\begin{enumerate}
\item $\g$ is arquimedian.
\item $\g$ has rank one.
\item $\g$ is isomorphic to a subgroup of $\R$. 
\end{enumerate}
\end{proposition}

\begin{proof}
The equivalence between (1) and (2) follows from Lemmas \ref{IArq} and \ref{IConv}. 

Clearly, (3) implies (1). Finally, if $\g$ is arquimedian, the choice of any positive $\be\in \g$ determines a unique embedding $\g\hk\R$ such that $\be\mapsto 1$. Indeed, any $\ga\in\g$ is mapped to the real number determined by the sequence of rational numbers $m/n$ such that $m\be\le n\ga$. Thus, (1) implies (3).
\end{proof}

%\subsection*{Maximal groups} \noindent{\bf Definition. }An ordered group $\g$ is \emph{maximal} if it admits no immediate extensions. More precisely, if any immediate extension of $\g$ is an isomorphism.\e \begin{theorem}Let $I$ be an ordered set and $(C_i)_{i\in I}$ a family of ordered groups of rank one, parameterized by $I$.The Hahn product $\left(\prod_{i\in I}C_i\right)_{\lx}$ is maximal.\end{theorem}

%\subsection*{Hahn's theorem and embedding into $\rlex$}\mbox{\null}

\noindent{\bf Definition. }An ordered group $\g$ is \emph{regular} if for all $i\in I=\pcv(\g)$, there exists a ring $\Z\subset A_i\subset \Q$ such that the component $C_i(\g)$ is free as an $A_i$-module.

\begin{theorem}[Hahn's theorem]
 Every regular ordered group $\g$ admits an immediate embedding to the Hahn product determined by the skeleton of $\g$.
\end{theorem}

%\begin{corollary}Every regular and maximal ordered group is isomorphic to the Hahn product determined by its skeleton.\end{corollary}

\subsection*{Functoriality of Hahn's embedding}

Let $\g$ be an arbitrary (not necessarily re\-gular) ordered group, with skeleton $\left(I;(C_i)_{i\in I}\right)$. The skeleton of $\gq$ is
$$
\left(I;(Q_i)_{i\in I}\right),\qquad Q_i=C_i\otimes_\Z\Q\ \mbox{ for all }i\in I.
$$

The canonical embedding $\g\hk \gq$ does not preserve the skeleton, but the embedding $\pcv(\g)\hk\pcv(\gq)$ described in Lemma \ref{rkEmbed}  is an isomorphism of ordered sets which we consider as a natural identification:
$$I=\pcv(\g)=\pcv(\gq).$$

By Hahn's theorem, there is a (non-canonical) immediate embedding 
$$
\gq\hooklongrightarrow\hq:=\left(\prod\nolimits_{i\in I}Q_i\right)_{\lx}.
$$

For each $i\in I$ we fix, once and for all, a positive element $1^i\in Q_i$. 
As shown in Proposition  \ref{rk1}, this choice determines an embedding $Q_i\hk\R$ of ordered groups, which sends our fixed element $1^i$ to the real number $1$.

We get an embedding $\hq \hk \rlex$ which obviously preserves the principal rank.

\begin{corollary}\label{embRlx}
The ordered group $\g$ admits an embedding $\g \hk \rlex$ which
induces a canonical identification $I=\pcv(\g)=\pcv(\rlex)$.
\end{corollary}

Let us analyze to what extent this embedding has a functorial behaviour.

Consider an embedding $\iota\colon\g\hk\La$ of ordered groups. 
Since $\La_\Q$ is divisible, Lemma \ref{MInDiv} shows that there exists a commutative diagram of embeddings 
\begin{equation}\label{divhull}
\as{1.2}
\begin{array}{ccc}
\La&\lra&\La_\Q\\
\uparrow&&\uparrow\\
\g&\lra&\gq
\end{array} 
\end{equation}

By Lemma \ref{rkEmbed}, $\iota$ induces an embedding
$$
I=\pcv(\gq)\stackrel{\iota}\hooklongrightarrow J=\pcv(\La)=\pcv(\La_\Q).
$$

Also, if the skeleton of $\La_\Q$ is $\left(J;(L_j)_{j\in J}\right)$, then, for each $i\in I$ there is a canonical embedding of abelian ordered groups
$Q_i \hk L_{\iota(i)}$.

In particular, $\iota$ induces a canonical embedding
$$
\H(\gq)\hooklongrightarrow\H(\La_\Q).
$$

For all $i\in I$, let $j=\iota(i)$, and denote by $1^j\in L_j$ the image of the positive element $1^i\in Q_i$ by the embedding $Q_i\hk L_j$. This choice  determines an embedding $L_j\hk\R$ such that the composition
$Q_i\hk L_j \hk \R$
is our fixed embedding $Q_i\hk\R$. 

Therefore, we have a commutative diagram of embeddings

\begin{equation}\label{hhnq}
\as{1.2}
\begin{array}{ccc}
\H(\La_\Q)&\lra&\R^J_{\lx}\\
\uparrow&&\uparrow\\
\hq&\lra&\rlex
\end{array} 
\end{equation}
The right-hand vertical mapping sends $(x_i)_{i\in I}\ \mapsto\ (y_j)_{j\in J}$, where
\begin{equation}\label{rirj}
y_j=
\begin{cases}
x_i,&\mbox{ if }j=\iota(i),\\
0,&\mbox{ if }j\not\in\iota(I).
\end{cases} 
\end{equation}

We would like to join the commutative diagrams (\ref{divhull}) and (\ref{hhnq}) into a commutative diagram of embeddings of ordered groups
\begin{equation}\label{embRlex}
\as{1.4}
\begin{array}{ccccccc}
\La&\lra&\La_\Q&\lra&\mathbb{H}(\La_\Q)&\lra&\R^J_{\lx}\\
\uparrow&&\uparrow&&\uparrow&&\uparrow\\
\g&\lra&\gq&\lra&\hq&\lra&\rlex
\end{array} 
\end{equation}

To this end,we need only to check that Hahn's immediate embeddings
$$
\gq\hooklongrightarrow\H(\gq),\qquad \La_\Q\hooklongrightarrow\H(\La_\Q),
$$
can be chosen in a compatible way, leading to a commutative diagram
\begin{equation}\label{choices}
\as{1.4}
\begin{array}{ccc}
\La_\Q&\hooklongrightarrow&\H(\La_\Q)\\
\uparrow&&\uparrow\\
\gq&\hooklongrightarrow&\hq 
\end{array} 
\end{equation}

To check this, we must review how the embedding $\gq\hk\H(\gq)$ is constructed.

The proof of Hahn's theorem for a divisible ordered group relies in the following result of B. Banaschewski, which makes use of Zorn's lemma.  

\begin{lemma}\label{bski}
 Let $V$ be a vector space over a field $K$. Consider a non-empty set $ \ss\subset\op{Subsp}(V)$ of subspaces of $V$.
 Then, there exists a mapping
$$
\op{cmpl}\colon \ss\lra\ \op{Subsp}(V)
$$
satisfying the following properties:
\begin{enumerate}
\item $\ V=W\oplus \op{cmpl}(W)$, \ for all $W\in\ss$.
\item $\ W\subset W'\;\imp\,\op{cmpl}(W)\supset\op{cmpl}(W')$, \ for all $\,W,W'\in\ss$. 
\end{enumerate}
\end{lemma}

By Banaschewski's lemma, we may choose complementary $\Q$-subspaces:
$$
\gq=H\oplus\op{cmpl}(H),\quad\mbox{for all }H\in \cv(\gq), 
$$
with a coherent behaviour with respect to inclusions, as indicated in condition (2). 

In particular, for all $i\in I$ we have projections:
$$
\pi_i\colon \gq\longtwoheadrightarrow H_i\longtwoheadrightarrow Q_i. 
$$
The projection $\g\twoheadrightarrow H_i$ depends on the choice of the complementary subspace of $H_i$. The projection $H_i\twoheadrightarrow Q_i=H_i/H_i^*$ is the canonical quotient mapping.

For simplicity, let us denote $\pi_i(\ga)=\ga_i\in Q_i$ for all $\ga\in\gq$, $i\in I$. 

In this way, we obtain a group homomorphism:
$$
\varphi\colon\gq\lra \prod_{i\in I}Q_i,\qquad \ga\ \longmapsto\ \left(\ga_i\right)_{i\in I},
$$
which is injective. In fact, if $\ga\in\gq$ is non-zero, the principal subgroup $H(\ga)$ generated by $\ga$ is non-zero too; thus,  $H(\ga)=H_i$ for some $i\in I$. The element $\ga_i\in Q_i$ is the class of $\ga$ modulo $H_i^*$.
Since $\ga$ generates $H_i$, we have  $\ga_i\ne0$.

The proof of Hahn's theorem ends by checking that $\varphi(\gq)\subset\H(\gq)$ and $\varphi$ preserves the ordering \cite[Sec. A]{Rib}.\e

Now, let us go back to diagram (\ref{choices}). The commutativity of the diagram is equivalent to the commutativity of the following diagram, for all $i\in I$:
$$
\as{1.4}
\begin{array}{ccccc}
\La_\Q&\stackrel{\pi_j(\La)}\hooklongrightarrow&H_j&\lra& L_j\\
\uparrow&&\uparrow&&\uparrow\\
\gq&\stackrel{\pi_i(\g)}\hooklongrightarrow&H_i&\lra&Q_i
\end{array} 
$$
where $\iota(i)=j\in J$, and $H_j$ is the convex subgroup of $\La_\Q$ generated by $\iota(H_i)$.

The right-hand diagram commutes because the vertical mappings are induced by $\iota$ and the horizontal mappings are canonical.

Thus, we need only to show that the left-hand diagram commutes. This amounts to choose the complementary subspaces of convex sugbroups 
in $\gq$ and $\La_\Q$ so that 
\begin{equation}\label{ok}
\iota(i)=j\ \imp\ \iota\left(\op{cmpl}_\g(H_i)\right)\subset \op{cmpl}_\La(H_j). 
\end{equation}

This is always possible. For instance, we may first apply Banaschewski's lemma to the set $\ss_\g=\cv(\gq)\subset\op{Subsp}(\gq)$ to consider coherent choices $\op{cmpl}_\g(\Delta)$ of complementary subspaces of all convex subgroups $\Delta\in\cv(\gq)$. Then, for all $H\in \cv(\La_\Q)$ we may consider the subspace
$$H\oplus H'\subset \La_\Q,\qquad H'=\iota\left(\op{cmpl}_\g(\Delta)\right),$$
where $\Delta$ is the maximal convex subgroup of $\gq$ such that $\iota(\Delta)\subset H$. 

Finally, we may apply  Banaschewski's lemma to the set $$\ss_\La=\left\{H\oplus H'\mid H\in \cv(\La_\Q)\right\}\subset \op{Subsp}(\La_\Q)$$ and choose complementary subspaces 
$$
\La_\Q=H\oplus H'\oplus\op{cmpl}(H\oplus H'),
$$
satisfying condition (2) of Lemma \ref{bski}. It is easy to check that the complementary subspaces $\op{cmpl}_{\La}(H):=H'\oplus\op{cmpl}(H\oplus H')$ satisfy condition (2) as well. These complementary subspaces $\op{cmpl}_{\La}(H)$ obviously satisfy (\ref{ok}).

Summing up, we have seen the functoriality of Hahn's construction.

\begin{lemma}\label{compatible}
Let $\g\hk\La$ be an embedding of ordered groups. Then, there exist choices of Hahn's immediate embeddings
$$
\gq\hooklongrightarrow\H(\gq),\qquad \La_\Q\hooklongrightarrow\H(\La_\Q),
$$
making diagram (\ref{embRlex}) commutative. 

\end{lemma}

%The immediate embedding $\gq\hk\hq$ given by Hahn's theorem relies on certain choices of $\Q$-bases of the components $Q_i$. Analogous choices for the components $D_i$ can be made in a compatible way with the embeddings $Q_i\hk L_i$. Hence, we may extend our commutative diagram above to a larger one:

%\noindent{\bf Example of an irregular ordered group. }Let $\g$ be the subgroup of $\Q^2_{\lx}$ generated by the elements$$(p_n^{-1},np_n^{-1}),\qquad n\in\N,\quad p_n\mbox{ the $n$-th prime number}.$$

%The only non-trivial proper subgroup of $\g$ is $H=\{0\}\times\Z$. The first component is $\g/H$ which is not free as an $A$-module for any ring $\Z\subset A\subset \Q$. Moreover, we cannot embed $\g$ into $(\g/H)\times H$ with the lexicographical order. 

We end this section with an auxiliary result which follows immediately from the above description of Hahn's embedding $\gq\hk\hq$.

\begin{lemma}\label{hahn}
For all $i\in I$, $q\in Q_i$, there exists an element $b_{i,q}\in\gq$ whose image in $\hq$ is of the form:
$$
b_{i,q}=(\cdots0\,0\, q \star\star\cdots).
$$ 
That is, $b_{i,q}=(b_j)_{j\in I}$, with $b_i=q$ and $b_j=0$ for all $j<i$.
\end{lemma}

\section{Small extensions of ordered groups}\label{secSmall}

%\section{Commensurable extensions}

%\subsection*{Commensurable extensions}

The \emph{rational rank} of an abelian group $G$ is the cardinality of any maximal subset of $\Z$-linearly independent elements in $G$:
 $$\rrk(G)=\dim_\Q(G\otimes_\Z\Q).$$ 
 
Thus, $\rrk(G)=0$ if and only if $G$ is a torsion group.\e

An extension of ordered groups $\g\hk\La$ is \emph{commensurable} if $\rrk(\La/\g)=0$.\e

The extension $\g\hk\gq$ is simultaneously the minimal divisible extension of $\g$ and the maximal commensurable extension of $\g$.

\begin{lemma}\label{MaxComm}
For any commensurable extension $\g\hk \La$, there exists a unique embedding of $\La$ into $\gq$ such that  the composition $\g\hk \La\hk\gq$ is the canonical embedding.
\end{lemma}

Two extensions $\ \g\hk \La$, $\ \g\hk \La'$ \ 
are said to be \emph{equivalent} if there is an isomorphism $\ \La\ism\La'$ \ of ordered groups fitting into a commmutative diagram:
$$
\as{1.2}
\begin{array}{ccc}
\La&&\\
\uparrow&\searrow&\\
\g&\lra&\La'
\end{array}
$$

By Lemma \ref{MaxComm}, every commensurable extension of $\g$ is equivalent to a unique subgroup of $\gq$.

\subsection{Small extensions}\label{subsecSmall}
For an arbitrary extension $\iota\colon \g\hk \La$, we denote by $$\g \hk \g^\com\subset \La$$ the maximal commensurable extension of $\g$ in $\La$; that is,
$$
\g^\com=\left\{\xi\in \La\mid n\xi\in\iota(\g),\ \mbox{for some}\ n\in\N\right\}.
$$

\noindent{\bf Definition.} We say that $\g\hk \La$ is a \emph{small extension} if $\La/\g^\com$ is a cyclic group.\e

Therefore, a small extension is either commensurable (if $\g^\com=\La$), or it has $\rrk(\La/\g)=1$ and the quotient  $\La/\g^\com$ is isomorphic to $\Z$.\e

This definition is motivated by the following result, which follows from the work by MacLane and Vaqui\'e.

\begin{theorem}\cite[Cor. 4.6]{MLV}\label{AllSmall}
Let $K$ be a field and let $\mu\colon \kx\twoheadrightarrow\gm\infty$ be a valuation on the polynomial ring $\kx$. Let $\g=\mu(K^*)$ be the value group  
of the restriction of $\mu$ to $K$. Then, $\g\subset\gm$ is a small extension of ordered groups.
\end{theorem}

%\begin{proof}The following mapping is a valuation on $\kx$:$$\mu_0\colon\kx\lra\gm\infty,\qquad \sum_{0\le s}a_sx^s\ \longmapsto\ \min\{\mu(a_sx^s)\mid 0\le s\}.$$

%With the terminology of \cite{MLV}, a celebrated result of MacLane-Vaqui\'e states that $\mu$ falls in one, and only one, of the following cases \cite[Thm. 4.3]{MLV}:  \e(a) \ There is a finite chain of mixed augmentations $\mu_0<\mu_1<\cdots<\mu_r=\mu$.\e(b) \  There is a finite chain of mixed augmentations $\mu_0<\mu_1<\cdots<\mu_r$ such that $\mu$ is the stable limit ofsome continuous family of augmentations of $\mu_r$.\e(c) \ There is a countably infinite chain of mixed augmentations $$\mu_0<\mu_1<\cdots<\mu_n<\cdots<\mu,$$ such that $\mu$ is the stable limit $\mu=\lim_{n\to\infty}\mu_n$. \e
 
% In cases (b) and (c), the extension $\g\subset\gm$ is commensurable \cite[Prop. 3.1]{MLV}. In case (a), $\gm$ is generated by $\g$ and a finite family $\ga_0,\dots,\ga_r\in\gm$, from which $\ga_0,\dots,\ga_{r-1}$ are commensurable over $\g$ \cite[Sec. 4.1]{MLV}.  If $\ga_r$ is commensurable over $\g$, then the extension $\g\subset\gm$ is commensurable. If $\ga_r\not\in\gq$, then  $\gm$ is obtained as a commensurable extension $\gen{\g,\ga_0,\dots,\ga_{r-1}}_\Z$ of $\g$, followed by an extension with cyclic quotient.\end{proof}

Let us exhibit a few examples of small and non-small extensions. Consider the following four extensions of $\g=\Z$:\e

\as{1.3}
\begin{tabular}{ll}
(a)\quad $\Z\subset \gen{1,\root3\of2}_\Z$,& \qquad(b)\quad $\Z\subset\Z[\root3\of2]$,\\
(c)\quad $\Z\hk \Q\times\Z,\quad m\mapsto(m,0)$,&\qquad(d)\quad $\Z\hk \Q\times\Z,\quad m\mapsto(0,m)$.
\end{tabular}\e
\as{1}

The extensions (a) and (b) preserve the rank, while (c) and (d) increase the rank by one. On the other hand, only (a) and (c) are small.

Clearly, small extensions increase the rank at most by one. Let us discuss in more detail this property.

%\subsection*{Equivalent extensions} 

%\subsection*{Extensions that increase the rank at most by one}

By Lemma \ref{rkEmbed}, any extension $\g\hk\La$ induces two embeddings of ordered sets
$$
\cv(\g)\hra\cv(\La),\qquad \pcv(\g)\hra\pcv(\La).
$$
The following well-known inequality is an easy consequence of Hahn's theorem:
\begin{equation}\label{rrrk}
\rrk(\La/\g)\ge \sharp \left(\pcv(\La)\setminus\pcv(\g)\right),
\end{equation}
where we identify $\pcv(\g)$ with its image in $\pcv(\La)$ under the above embedding.

Lemmas \ref{ConvI} and \ref{IConv} describe how the sets $\cv(\g)$ and $\pcv(\g)$ determine one to each other. From this relationship it is easy to deduce that
$$\sharp \left(\pcv(\La)\setminus\pcv(\g)\right)=0\sii \sharp \left(\cv(\La)\setminus\cv(\g)\right)=0$$
$$\sharp \left(\pcv(\La)\setminus\pcv(\g)\right)=1\sii \sharp \left(\cv(\La)\setminus\cv(\g)\right)=1$$\vskip0.3cm

\noindent{\bf Definition.} We say that the extension $\ \g\hk\La\ $ \emph{increases the rank at most by one} if $$\sharp \left(\pcv(\La)\setminus\pcv(\g)\right)\le1.$$

If $\sharp \left(\pcv(\La)\setminus\pcv(\g)\right)=0$ we say that $\g\hk\La$ \emph{preserves the rank}.

If $\sharp \left(\pcv(\La)\setminus\pcv(\g)\right)=1$ we say that $\g\hk\La$ \emph{increases the rank by one}.
\vskip0.3cm

For instance, it follows from (\ref{rrrk}) that the extension $\g\hk\gq$ preserves the rank.\e

\noindent{\bf Caution!}  If $\rk(\g)$ is finite, there is no ambiguity in these concepts. However, if $\rk(\g)$ is infinite, this terminology abuses of language. 
If $\g\hk\La$ preserves the rank, then obviously $\rk(\g)=\rk(\La)$, but the converse is not true.

For instance, $\N_0=\{0\}+\N$ is isomorphic to $\N$ as an ordered set; hence, the ordered groups $\R^{\N}_{\lx}$ and  $\R^{\N_0}_{\lx}$ have the same rank. However, the natural embedding $\R^{\N}_{\lx}\hk\R^{\N_0}_{\lx}$ increases the rank by one.\e

\begin{lemma}\label{small<=1}
Every small extension $\g\hk\La$ increases the rank at most by one 
\end{lemma}

\begin{proof}
Since $\rrk(\La/\g)\le1$, the statement follows from the inequality (\ref{rrrk}). 
\end{proof}

\subsection*{Small subextensions of a fixed universe}

From now on, we fix an extension $\g\hk U$ of ordered groups, and we identify $\g$ with its image in $U$.

%For any $\ga\in U$, the subgroup $\gen{\g^\com,\ga}$ generated by $\ga$ over $\g^\com$ is a small extension of $\g$ in $U$. 

We are not aiming at a classification of the small extensions of $\g$ in $U$. Rather, in view of the applications to valuation theory, we are interested in the classification of the elements in $U$ by a certain equivalence relation.\e

\noindent{\bf Definition.} We say that $\be,\ga\in U$ are \emph{$\g$-equivalent} if there exists an isomorphism of ordered groups
$$
\ggb\iso\gga,
$$
which acts as the identity on $\g$ and sends $\be$ to $\ga$. 

In this case, we write $\be\sim_\sme\ga$ if the base group $\g$ is clear from the context.

We denote by $[\be]_\sme\subset U$ the class of $\be$.
\e

By Lemma \ref{MaxComm}, a grup homomorphism $\gen{\g^\com,\be}\to U$ acting as the identity on $\g$, necessarily acts as the identity on $\g^\com$. 
This justifies the following result.
%The next result follows immediately from this fact.

\begin{lemma}\label{bebe}
\begin{enumerate}
\item Two elements $\be,\ga\in U$ are $\g$-equivalent if and only if they are $\g^\com$-equivalent. 
\item If $\be\in \g^\com$, then $[\be]_\sme=\{\be\}$. 
\end{enumerate}
\end{lemma}

Let $U^{\incom}=U\setminus \g^\com$ be the subset of incommensurable elements over $\g$. There is an easy criterion to decide when two elements in $U^{\incom}$ are $\g$-equivalent.

\begin{lemma}\label{criteri}
Let $\be,\ga\in U^{\incom}$ with $\be<\ga$. Then, $\be$ and $\ga$ are $\g$-equivalent if and only if there is no $b\in \g^\com$ such that $\be<b<\ga$.  
\end{lemma}

\begin{proof}
All elements in the subgroup $\ggb$ may be written in a unique way as $$a+m\be, \quad\mbox{ with }a\in\g, \ m\in\Z.$$
 
Hence, for all $\be,\ga\in U^{\incom}$ there is a unique group isomorphism $$h\colon \ggb\iso\gga$$ acting as the identity on $\g$ and sending $\be$ to $\ga$.
We have $\be\sim_\sme\ga$ if and only if this homomorphism $h$ preserves the ordering.

Suppose that $\be<b<\ga$, for some  $b\in \g^\com$. Then, $h$ does not preserve the ordering, because $\ga=h(\be)>b=h(b)$. Thus, $\be$ and $\ga$ are not $\g$-equivalent.

Conversely, suppose that there is no $b\in \g^\com$ such that $\be<b<\ga$. Let us check that the homomorphism $h$ preserves the ordering.

For all $a+m\be\in\ggb$, we clearly have,
\begin{equation}\label{unequal}
a+m\be>0 \sii 
\begin{cases}
m=0,\quad a>0,\ \mbox{ or}\\
m>0,\quad -a/m<\be,\ \mbox{ or}\\
m<0,\quad -a/m>\be.
\end{cases}
\end{equation}

If $m\ne 0$, then $-a/m$ belongs to $\g^\com$. By our assumption, 
$$
-a/m<\be\ \sii\ -a/m<\ga.
$$
Hence, the conditions of the right-hand side of (\ref{unequal}) are satisfied if we replace $\be$ with $\ga$. Therefore, $a+m\ga>0$, and this proves that $h$ preserves the ordering.
\end{proof}

Our aim in the following sections is to find explicit computations of the quotient set $U/\!\sim_\sme$ for some concrete ordered groups $U$.

\subsection{Small extensions that preserve the rank}\label{secSameRank}

Let $\g$ be an ordered group and let $\left(I;(Q_i)_{i\in I}\right)$ be the skeleton of $\gq$. 

For all $i\in I$ we fix an embedding $Q_i\hk\R$ of ordered groups. For all $q\in Q_i$ we use the same symbol $q\in\R$ to denote the image of $q$ by this embedding. 

By Corollary \ref{embRlx}, there is a rank-preserving embedding 
$$
\g\ \hra\ \gq\ \hra\ \hq\ \hra\ \rlex.
$$

This extension $\g\hk\rlex$ is maximal among all rank-preserving extensions of $\g$.

\begin{lemma}\label{MaxEqRk}
For any rank-preserving extension $\g\hk \La$, there exists an embedding $\La\hk\rlex$ fitting into a commutative diagram
$$
\as{1.2}
\begin{array}{ccc}
\La&&\\
\uparrow&\searrow&\\
\g&\lra&\rlex
\end{array}
$$
\end{lemma}

\begin{proof}
By hypothesis, we have an isomorphism of ordered sets:
$$
\pcv(\g)=I\iso J=\pcv(\La).
$$

By Lemma \ref{compatible}, the immediate embeddings $\gq\hk\hq$, $\La_\Q\hk\H(\La_\Q)$ can be chosen in a compatible way, leading to a commutative diagram of embeddings
$$%\begin{equation}\label{embRlex}
\as{1.4}
\begin{array}{ccccccc}
\La&\lra&\La_\Q&\lra&\mathbb{H}(\La_\Q)&\lra&\R^J_{\lx}\\
\uparrow&&\uparrow&&\uparrow&&\uparrow\\
\g&\lra&\gq&\lra&\hq&\lra&\rlex
\end{array} 
$$%\end{equation}
By equation (\ref{rirj}), the right-hand vertical mapping is an isomorphism.
\end{proof}\e

\noindent{\bf Caution!} The embedding $\La\hk\rlex$ is not necessarily unique. Thus, every rank-preserving extension of $\g$ is equivalent to some subextension of $\g\hk\rlex$, but not to a unique one!  \e

For instance, if $\be,\ga\in\rlex$ are two different incommensurable elements (over $\g$) which are $\g$-equivalent, then the subgroups $\ggb$ and $\gga$ are equivalent, but they may be different.

Our aim is to find a canonical system of representatives of $\rlex/\!\sim_\sme$.
Clearly, $$\left(\rlex\right)^\com=\gq,$$
and the classes of commensurable elements are computed in Lemma \ref{bebe}.

Thus, we focus on the computation of classes of incommensurable elements.\e

%\subsection*{Canonical system of representatives of $\left(\rlex\right)^{\incom}/\!\sim_\sme$}

For any initial segment $S\in\inii$, consider the canonical projection
$$
\pi_S\colon \rlex\lra\R^S_{\lx},\qquad \be=(\be_i)_{i\in I}\longmapsto \be_S=(\be_i)_{i\in S}. 
$$
This is a homomorphism of ordered groups, admitting a section
$$
\iota_S\colon \R^S_{\lx}\hra\rlex,\qquad \rho=(\rho_i)_{i\in S}\longmapsto \iota_S(\rho)=(\rho\mid 0), 
$$
where $(\rho\mid 0)$ has the obvious meaning.\e

\noindent{\bf Definition.} An element $\rho\in\R^S_{\lx}$ is said to be \emph{commensurable} over $\g$ if there exists $b\in\gq$ such that $b_S=\rho$.

%An element $\rho\in\R^S_{\lx}$ is said to be \emph{rationally incommensurable} if it is not commensurable and there exists $\be\in\hq$ such that $\be_S=\rho$.

%An element $\be\in\rlex$ is said to be \emph{irrationally incommensurable} if it is neither commensurable nor rationally commensurable.\e

\begin{lemma}\label{SameClass}
 For any $\be\in\left(\rlex\right)^{\incom}$ and any $S\in\inii$ such that $\be_S$ is incommensurable, we have $\be\sim_\sme(\be_S\mid0)$. 
\end{lemma}

\begin{proof}
Suppose that $\be<(\be_S\mid0)$. Any $\ga\in\rlex$ such that $\be<\ga<(\be_S\mid0)$ satisfies necessarily $\ga_S=\be_S$. Since this element is incommensurable, $\ga$ cannot belong to $\gq$. By the criterion of Lemma \ref{criteri}, 
$\be\sim_\sme(\be_S\mid0)$.

If $\be>(\be_S\mid0)$, the argument is completely analogous. 
\end{proof}

For the construction of a canonical system of representatives of $\left(\rlex\right)^{\incom}/\!\sim_\sme$ it suffices to consider inside each class $[\be]_\sme$ the element having minimal support.\e

\noindent{\bf Definition.} A \emph{minimal incommensurable} element is any $\be\in\rlex$ for which  there exists a minimal initial segment $S\in\inii$ satisfying: 
$$\be_S \ \mbox{ is incommensurable \ and \ }\be=(\be_S\mid0).
$$
We denote by $\eqr(\g)$ be the set of minimal incommensurable elements in $\rlex$. 

\begin{theorem}\label{MinClass}
The set of minimal incommensurable elements is a system of representatives of  $\left(\rlex\right)^{\incom}/\!\sim_\sme$.
\end{theorem}

The proof of this theorem follows from Lemmas \ref{MinInequiv} and \ref{MinExists} below.

%\begin{lemma}\label{WellOrd}For any well-ordered set $J$, there is an isomorphism of ordered sets:$$J\infty\lra \ini(J),\qquad j\ \longmapsto\ J_{<j}$$In particular, $\ini(J)$ is a well-ordered set.\end{lemma}

%\begin{proof}Clearly, this mapping strictly preserves the ordering:$$j<k\ \imp\ J_{<j}\subsetneq J_{<k}.$$Let us check that the mappping is onto.Since $J\infty$ is well-ordered, for any $S\in\ini(J)$, there exists $j_0=\min(J\infty\setminus S)$. Clearly, $S=J_{<j_0}$.\end{proof}

\begin{lemma}\label{MinInequiv}
The minimal incommensurable elements are pairwise $\g$-inequivalent. 
\end{lemma}

\begin{proof}
Let $\be=(\be_S\mid0),\,\ga=(\ga_T\mid0)$ be two minimal incommensurable elements such that $\be\ne\ga$, and take $j=\min(\supp(\be-\ga))$. Assume for instance $S\subset T$.
 
If $j>S$, we have $\be_S=\ga_S$. Since $\be\ne\ga$, this implies $S\subsetneq T$, and this contradicts the minimality of $T$. Therefore, $j\in S$.

 Let $R=I_{<j}\in\inii$. Since $R\subsetneq S$, $\ga_R=\be_R$ is commensurable. Let $b=(b_i)_{i\in I}\in\gq$ such that $\ga_R=\be_R=b_R$.

Also, if (for instance) $\be_j<\ga_j$,  there exists $q\in Q_j$ such that $\be_j<q<\ga_j$. Now, consider the element $b_{j,q-b_j}\in\gq$ defined in Lemma \ref{hahn}. The element $c=b+b_{j,q-b_j}\in\gq$ satisfies $\be<c<\ga$. By Lemma \ref{criteri}, $\be$ and $\ga$ are not $\g$-equivalent.
\end{proof}

\begin{lemma}\label{MinExists}
For any $\be\in\left(\rlex\right)^{\incom}$, the subset of $\inii$ formed by the initial segments $T$ such that $\be_T$ is incommensurable contains a minimal element $S$. 

Moreover, 
 $(\be_S\mid0)$ is a minimal incommensurable element and $[\be]_\sme=\pi_S^{-1}(\be_S)$. 
\end{lemma}

\begin{proof}
Let $J=\supp(\be)$, which is a well-ordered subset of $I$. Hence, $J\infty$ is a well-ordered subset of $I\infty$. Consider the subset
 $$
A=\left\{j\in J\infty\mid \be_{I_{<j}} \mbox{ is incommensurable}\right\}\subset J\infty.$$

Since $\be$ is incommensurable, $\infty\in A$, so that $A\ne\emptyset$. Thus, there exists $i=\min(A)$, and $S=I_{<i}$ is the minimal initial segment such that $\be_S$ is incommensurable.

By definition, $(\be_S\mid0)$ is a minimal incommensurable element. By Lemma \ref{SameClass},  $\be\sim_\sme (\be_S\mid0)$ and, more generally, $\pi_S^{-1}(\be_S)\subset [\be]_\sme$.

Finally, take any $\ga\in [\be]_\sme$. Let $T$ be the minimal initial segment of $I$ such that $\ga_T$ is incommensurable. Since $(\ga_T\mid0),\, (\be_S\mid0)\in [\be]_\sme=[\ga]_\sme$, Lemma  \ref{MinInequiv} shows that $T=S$ and $\ga_S=\be_S$. Hence, $\ga$ belongs to  $\pi_S^{-1}(\be_S)$.
\end{proof}

%This ends the proof of Theorem \ref{MinClass}.\e

\noindent{\bf Definition.} 
We define the \emph{equal-rank closure} of $\g$ as the totally ordered set
$$
\grr=\gq\,\sqcup\,\eqr(\g)\subset\R^I_{\lx},
$$
which is a canonical system of representatives of $\rlex/\!\sim_\sme$.\footnote{While $\gq$ is defined as $\g\otimes_\Z\Q$, let us remark that $\grr$ has nothing to do with $\g\otimes_\Z\R$.}
%\e

\subsection*{Rational and irrational incommensurable classes in $\grr$}\mbox{\null}\e
%The set $\eqr(\g)$, canonical system of representatives of \ $\left(\rlex\right)^{\incom}/\!\sim_\sme$, \splits in a natural way into the disjoint union of two subsets.\e

\noindent{\bf Definition.} Let $\be\in\left(\rlex\right)^{\incom}$. The class $[\be]_\sme$ is said to be  \emph{rationally incommensurable} if  there exists $\ga\in\hq$ such that $\be\sim_\sme\ga$.

Otherwise, the class $[\be]_\sme$ is said to be \emph{irrationally incommensurable}.\e

Thus, we may split the set $\eqr(\g)$ into the disjont union of two subsets 
$$
\eqr(\g)=\eqrat(\g)\,\sqcup\,\eqri(\g),
$$
which represent the rationally and irrationally incommensurable classes, respectively.

It is easy to check that \ $\eqrat(\g)=\eqr(\g)\cap\hq$ \ is a system of representatives of $\ \left(\hq\setminus\gq\right)/\!\sim_\sme$.

Let us describe the subset $\eqri(\g)$ in more detail.

\begin{lemma}\label{Srat}
Let $\be=(\be_S\mid0)\in\eqr(\g)$ be a minimal incommensurable element in $\rlex$. The following conditions are equivalent.
\begin{enumerate}
\item $\be$ does not belong to $\hq$.
\item The initial segment $S$ contains a maximal element. 
\end{enumerate}
\end{lemma}

\begin{proof}
Suppose that  $\be=(\be_j)_{j\in I}\not\in\hq$, and let $J=\supp(\be)$. By our assumption, $$J^0:=\{j\in J\mid \be_j\not\in Q_j\}\ne\emptyset.$$
Since $J$ is well-ordered, there exists $i=\min\left(J^0\right)$. Since $\be_{i}\ne0$, we have $i\in S$. 
%$$R_0=I_{<j_0}\subsetneq R=I_{\le j_0}\subset S.$$

Let  $R=I_{\le i}\subset S$. 
Since $\be_R$ is incommensurable, %and $\be_{R_0}$ is commensurable, 
we must have $R=S$ by the minimality of $S$. Hence, $i=\max(S)$.

Conversely, suppose that $S$ contains a maximal element $i$, and let $T=I_{<i}\subsetneq S$. By hypothesis, $\be_{T}$ is commensurable, so that there exists $b\in\gq$ such that $\be_{T}=b_{T}$. In particular, $\be_j\in Q_j$ for all $j<i$. This implies that $\be_{i}\not\in Q_{i}$, so that $\be$ does not belong to $\hq$. 
Indeed, if $\be_{i}\in Q_{i}$, we may consider the element  $c=b_{i,\be_{i}-b_{i}}\in\gq$ described in Lemma \ref{hahn}. The element $b+c\in\gq$ satisfies $(b+c)_S=\be_S$, which is a contradiction.
\end{proof}

By Lemma \ref{Srat}, we have $\eqri(\g)=\bigcup\nolimits_{i\in I}\eqri(\g)_i$, where
$$
\eqri(\g)_i:=\left\{\left(b\mid q\mid0\right)\mid b\in \R^{I_{<i}}_{\lx}\mbox{ commensurable},\ q\in\R\setminus Q_i\right\}.
$$

\noindent{\bf Caution!} The set $\eqri(\g)$ does not represent all classes in $\left(\rlex\setminus\hq\right)/\!\sim_\sme$. In this latter set we may have rationally incommensurable classes. That is, there may exist elements $\be\in \rlex\setminus\hq$ such that the class $[\be]_\sme$ contains elements in $\hq$.\e

Before giving some examples, let us emphasize a relevant observation, which is an immediate consequence of the fact that $\gq=\hq$ if $\g$ has finite rank. 

\begin{lemma}\label{FiniteRank}
If $\rk(\g)$ is finite, then \ $\eqrat(\g)=\emptyset$.  
\end{lemma}

\subsection*{Examples}\mbox{\null}\e

(0) \ $\g=\{0\}$.
$$I=\emptyset,\quad \gq=\grr=\{0\}.$$

(1) \ $\rk(\g)=1$.
$$
\eqrat=\emptyset,\qquad \eqri=\R\setminus\gq,\qquad \grr=\R.
$$

(2) \ $\g=\R^2_{\lx}$.
$$
\eqrat=\eqri=\emptyset,\qquad \grr=\g.
$$

(3) \ $\g=\Q^2_{\lx}$.
$$
\eqrat=\emptyset,\qquad \eqri=\left\{(x,0)\in\R^2\mid x\in\R\setminus\Q\right\}\,\sqcup\,\left(\Q\times(\R\setminus\Q)\right).
$$
$$
\grr=\R^2_{\lx}\setminus\left\{(x,y)\in\R^2\mid x\not\in\Q,\ y\ne0\right\}.
$$

(4) \ $\g=\R^{(\N)}$.
$$
\eqrat=\R^{\N}\setminus \R^{(\N)},\qquad \eqri=\emptyset,\qquad \grr=\R^{\N}_{\lx}=\R^{\N}.
$$

(5) \ $\g=\Q^{(\N)},\qquad \rlex=\R^{\N}$.
$$
\eqrat=\Q^{\N}\setminus \Q^{(\N)},\qquad \eqri=\bigcup\nolimits_{i\in\N}\left(\Q^{i-1}\times \left(\R\setminus\Q\right)\times \{0\}^{\N_{>i}}\right).
$$
$$
\grr=\Q^{\N}\,\sqcup\,\eqri.
$$

%(6) \ $\g=\R^{(I)}$, for $I=\N+\N\simeq \left(\{1,2\}\times\N\right)_{\lx}$.$$\rlex=\left(\R^\N\times \R^\N\right)_{\lx},\qquad \eqri=\emptyset.$$$$\eqrat=\left\{\right(x,0)\mid x\in \R^{\N}\setminus \R^{(\N)}\}\,\sqcup\,\left\{\right(x,y)\mid x\in\R^{(\N)},\ y\in \R^{\N}\setminus \R^{(\N)}\}.$$$$\grr=\left(\R^\N\times \R^\N\right)_{\lx}\setminus\left\{(x,y)\in\R^\N\times \R^\N\mid x\not\in\R^{(\N)},\ y\ne0\right\}.$$\e

%\subsection{Small extensions that increase the rank by one}

\subsection{One-added-element embeddings of ordered sets}\label{secOneAdd}\mbox{\null}\e

\noindent{\bf Definition.} An embedding of (totally) ordered sets $$\iota\colon I\hra J$$ is a \emph{one-added-element} embedding if  $J\setminus \iota(I)$ is a one-element subset of $J$.\e

Let $I$ be an ordered set. For any $S\in\inii$, consider the ordered set 
\begin{equation}\label{IS}
I_S=S+\{i_S\}+S^c,
\end{equation}
where $S^c=I\setminus S$ is the complementary subset of $S$ in $I$.

The natural embedding $I\hk I_S$ is a one-added-element embedding. Also, every one-added-element embedding $I\hk J$ is isomorphic to $I_S$ for a unique $S\in\inii$.

More precisely, there is a unique $S\in\inii$ and a unique isomorphism $I_S\ism J$ of odered sets, fitting into a commutative diagram
$$
\as{1.2}
\begin{array}{ccc}
I_S&&\\
\uparrow&\searrow&\\
I&\hra&J
\end{array}
$$

For instance, for $I=J=\N$ and the one-added-embedding
$$\N\hra\N,\qquad n\longmapsto n+1,$$
we have $S=\emptyset$, $I_S=\{0\}+\N$ and the isomorphism $I_S\ism J$ maps $n\mapsto n+1$ for all $n\ge0$.  

\subsection*{Universal construction}

Let $I$ be an ordered set. Consider the set
$$
\I:=I\cup\left\{i_S\mid S\in\inii\right\}.
$$
We may consider a natural total ordering determined by

\begin{enumerate}
\item For all $S\in\inii$, the restriction of the ordering to $I_S=I\cup\{i_S\}$ is the ordering considered in (\ref{IS}). 
\item $i_S<i_T\ \sii\ S\subsetneq T,\quad $for all $S,T\in\inii$.
\end{enumerate}\e

This ordered set is called the \emph{one-added-element hull} of $I$. It satisfies an obvious universal property.

\begin{lemma}\label{OAEUniv}
For any  one-added-element embedding $I\hk J$ of ordered sets, there exists a unique embedding $J\hk \I$ fitting into a commutative diagram
$$
\as{1.2}
\begin{array}{ccc}
J&&\\
\uparrow&\searrow&\\
I&\hra&\I
\end{array}
$$
The image of $J$ in $\I$ is $I_S$ for a unique $S\in\inii$.
\end{lemma}

\subsection{Small extensions that increase the rank by one}\label{secLargerRank}

%Let $\g\hk\La$ be an extension of ordered groups.
We keep with the notation of sections \ref{secSameRank} and \ref{secOneAdd}.

\begin{lemma}\label{SemiUniv}
If the extension  $\g\hk\La$ increases the rank by one, there is a unique $S\in\inii$ and an embedding $\La\hk\R^{I_S}_{\lx}$ fitting into  commutative diagram:
$$
\as{1.4}
\begin{array}{ccccccc}
&&\La&&&&\\
&\nearrow&&\searrow&&&\\
\g&\hra&\rlex&\hra&\R^{I_S}_{\lx}&\hra&\rii
\end{array} 
$$
\end{lemma}

\begin{proof}
The initial segment $S$ is uniquely determined by the condition $\pcv(\La)\simeq I_S$.

Then, the proof  follows immediately from Lemma \ref{OAEUniv}.
\end{proof}

\noindent{\bf Caution!} For all $S\in\inii$, the subextensions of $\g\hk \R^{I_S}_{\lx}$ increase the rank at most by one. However, the extension $\g\hk \rii$ admits subextensions yielding a much larger increase of the rank. We consider the ordered group $\rii$ only to make it clear that the union of ordered groups 
$$\bigcup\nolimits_{S\in\inii}\R^{I_S}_{\lx}\,\subset\, \rii$$
has a natural total ordering.\e

This section is devoted to classify under the equivalence relation $\sim_\sme$ the elements $\ \be\in \bigcup\nolimits_{S\in\inii}\R^{I_S}_{\lx}\ $ such that the extension $\g\subset \ggb$ increases the rank by one.

Note that the concept of minimal incommensurable element and Lemmas \ref{SameClass}, \ref{MinExists} hold in our larger groups $\R^{I_S}_{\lx}$.

%The application of the criterion of Lemma \ref{criteri}, leads to the computation of a set of representatives of the quotient set  $$\left(\bigcup\nolimits_{S\in\inii}\R^{I_S}_{\lx}\setminus \rlex\right)\Big/\!\sim_\sme$$

%To start with, we must determine what elements $\ga\in \left(\bigcup\nolimits_{S\in\inii}\R^{I_S}_{\lx}\right)\setminus \rlex$ determine an extension $\g\subset \gga$ which increases the rank by one. 

%\noindent{\bf Example.} The group $\g=\Q$ has $I=\{1\}$ and $\grr=\rlex=\R$. The total group $H_1=\g$ is the unique non-trivial principal convex subgroup.

%For $S=I$, we have $I_S=\{1,2\}$, where we denote $2=i_S$ for simplicity. Thus, $\R^{I_S}_{\lx}=\R^2_{\lx}$, and the embedding $\rlex\hk \R^2_{\lx}$ has image $\R\times\{0\}$. The new principal subgroup of $\R^2_{\lx}$ is $\{0\}\times\R$.

%Let $\be=(\sqrt2,\sqrt3)\in\R^{I_S}_{\lx}\setminus\rlex$. The extension of $\g$ determined by $\be$ is$$\ggb=\left\{\left(q+m\sqrt2,m\sqrt3\right)\mid q\in\Q,\ m\in\Z\right\}.$$This group has rank one because $\left(\{0\}\times \R\right)\cap \ggb$ is the trivial group.

%Alernatively, the minimal incommensurable element in the class $[\be]_\sme$ is $\ga=\left(\sqrt2,0\right)$, which belongs to $\rlex$. Thus, $\ggb\simeq\gga$ has rank one.\bs

\begin{lemma}\label{CritIncr}
Let $\be=(\be_j)_{j\in I}\in \R^{I_S}_{\lx}$. The following conditions are equivalent.
\begin{enumerate}
\item The extension $\g\hk\ggb$ increases the rank by one.
\item The minimal incommensurable element in $[\be]_\sme$ does not belong to $\rlex$.
\item The minimal incommensurable element in $[\be]_\sme$ is $\left(\be_{S+\{i_S\}}\mid0\right)$.
\item The subgroup $\ggb$ contains an element $\ga=(\ga_i)_{i\in I_S}$ such that $$\ga_{i_S}\ne0,\quad\mbox{ and }\quad\ga_i=0 \ \mbox{ for all }i\in S.$$
\end{enumerate}
\end{lemma}

\begin{proof}
Let $T\subset I_S$ be the initial segment such that $(\be_T\mid0)$ is the minimal incommensurable element in $[\be]_\sme$. Denote $S'=S+\{i_S\}$.\e

(1) $\Rightarrow$ (2).  By Lemma \ref{SameClass}, the condition  $(\be_T\mid0)\in\rlex$ implies that the group $\ggb\simeq\gen{\g,(\be_T\mid0)}$ does not increase the rank. This contradicts (1).\e

(2) $\Rightarrow$ (3). The condition $(\be_T\mid0)\not\in\rlex$ is equivalent to $\be_{i_S}\ne0$. Thus, $i_S\in T$ and $S'\subset T$. Since $\be_{S'}$ is incommensurable, we have $T=S'$ by the minimality of $T$.\e

(3) $\Rightarrow$ (4). Since $i_S=\max(S')$, the minimality of $S'$ implies $\be_{i_S}\ne0$ and $\be_S$ commensurable. Let $b\in\gq$ such that $\be_S=b_S$.
The element $\ga=\be-b$ belongs to $\gen{\gq,\be}$ and satisfies the conditions of (4). Hence, there exists $m\in\N$ such that $m\ga$ belongs to $\ggb$ and satisfies the conditions of (4). \e

(4) $\Rightarrow$ (1). The chain of non-zero principal convex subgroups of $\rlex$ is 
$$
\left(H_i\right)_{i\in I},\qquad H_i=\left\{(\be_j)_{j\in I}\mid \be_j=0,\;\forall\,j<i\right\}.
$$
The chain of non-zero  principal convex subgroups of $\g$ is $\left(H_i\cap\g\right)_{i\in I}$. These groups induce the chain $\left(H_i\cap\ggb\right)_{i\in I}$ of non-zero principal subgroups of $\ggb$.

Let $H$ be the principal subgroup of $\R^{I_S}_{\lx}$ generated by an element $\ga\in \ggb$ satisfying (4). We claim that $H\cap\ggb$ is  a new principal subgroup of $\ggb$. Indeed, the condition $\ga_{i_S}\ne0$ implies that  $\ga$ does not belong to the smaller subgroup  $\bigcup_{i> i_S}\left(H_{i}\cap\ggb\right)$. The condition  $\ga_i=0$ for all $i\in S$ implies that $\ga$ does not generate any larger subgroup $H_{i}\cap\ggb$ for $i\in S$.
\end{proof}

We may now proceed to compute a system of representatives of the subset of $\ \R^{I_S}_{\lx}/\!\sim_\sme\ $ formed by the classes that increase the rank.

By Lemma \ref{CritIncr}, the minimal incommensurable elements in these classes are
$$
\be=\be_{S,a,q}=\left(a\mid q\mid 0\right),\qquad a\in\R^S_{\lx} \mbox{ commensurable},\quad q=\be_{i_S}\in \R^*,
$$
for an arbitrary $S\in\inii$.

Nevertheless, we cannot proceed as in section \ref{secSameRank} because Lemma \ref{MinInequiv} fails. There are minimal incommensurable elements in $\R^{I_S}_{\lx}\setminus \rlex$ which are equivalent. 

\begin{lemma}\label{MinEquiv}
We have $\ \be_{S,a,q}\sim_\sme\be_{T,b,p}$ \ if and only if
\begin{equation}\label{pq>0}
S=T, \qquad a=b,\qquad pq>0. 
\end{equation}
\end{lemma}

\begin{proof}
If the conditions of (\ref{pq>0}) are satisfied, we have $\be_{S,a,q}\sim_\sme\be_{S,a,p}$ by Lemma \ref{criteri}. Indeed, for any $\ga=(\ga_i)_{i\in I_S}\in\R^{I_S}_{\lx}$, the condition $\be_{S,a,q}<\ga<\be_{S,a,p}$ implies $\ga_S=a$ and $q<\ga_{i_S}<p$. Since we are asuming that $p$ and $q$ have the same sign, this implies $\ga_{i_S}\ne0$, so that $\ga$ cannot be commensurable over $\g$.

Conversely, suppose that $\be_{S,a,q}\sim_\sme\be_{T,b,p}$. Arguing as in the proof of Lemma \ref{MinInequiv}, we conclude that $S=T$ and $a=b$. 

Finally, suppose that $p$ and $q$ have a different sign; for instance, $q<0<p$. Then, any $\ga=(\ga_i)_{i\in I}\in\gq$ such that $\ga_S=a$ satisfies  $\be_{S,a,q}<\ga<\be_{S,a,p}$, because $\ga_{i_S}=0$. This is impossible by Lemma \ref{criteri}; thus, $p$ and $q$ must have the same sign.  
\end{proof}

As a consequence, the classes in \ $\left(\bigcup\nolimits_{S\in\inii}\R^{I_S}_{\lx}\right)\Big/\!\sim_\sme$ \ which increase the rank are
represented by the set
$$
\incrg=\bigcup_{S\in\inii}\left\{b^-,\,b^+\mid b\in\R^S_{\lx}\mbox{ commensurable}\right\},
$$
where we define
$$
b^-=\be_{S,b,-1}=\left(b\mid -1\mid 0\right),\qquad b^+=\be_{S,b,1}=\left(b\mid 1\mid 0\right).
$$
If $b^{\pm}=(b_i)_{i\in I_S}$, note that $b_{i_S}=\pm1$ and $b_j=0$ for all $j>i_S$.\e

The elements corresponding to $S=\emptyset$ deserve a special notation
$$
-\infty=\be_{\emptyset,-1}=(-1\mid0),\qquad \infty^-=\be_{\emptyset,1}=(1\mid0).
$$
The notation for $\infty^-$ is motivated by the fact that this element is the immediate predecessor of $\infty$ in the set $\gsme\infty$. \e

\noindent{\bf Definition.} The \emph{small-extensions closure} of $\g$ is the ordered set
$$
\gsme=\grr\,\sqcup\,\incrg,
$$
with the ordering induced by $\rii$.

This set is a system of representatives of \ $\left(\bigcup_{S\in\inii}\R^{I_S}_{\lx}\right)\Big/\!\sim_\sme$.\e

The next result follows immediately from Lemmas \ref{small<=1}, \ref{bebe}, \ref{MaxEqRk} and  \ref{SemiUniv}.

\begin{proposition}\label{smallSme}
Let $\g\hk\La$ be a small extension of $\g$, and let $\ga\in\La$ such that $\La=\gen{\g^{com},\ga}$. Let $\g^\com\!\ism\Delta\subset\gq$ be the canonical embedding of $\g^{com}$ into $\gq$.   

Then, for a  unique $\be\in\gsme$ there exists an isomorphism  of ordered groups $$\La\iso\gen{\Delta,\be},$$ sending $\ga$ to $\be$, and whose restriction to $\g^{com}$ is the canonical isomorphism $\g^{com}\!\ism \Delta$.    
\end{proposition}

\section{Basic properties of $\gsme$}

\subsection{Density of $\gq$ and canonical ordering}

\begin{lemma}\label{density}
Let $\be,\ga\in\bigcup_{S\in\inii}\R^{I_S}_{\lx}$ such that $\be<\ga$ and $[\be]_\sme\ne [\ga]_\sme$. Then,
\begin{enumerate}
\item There exists $q\in\gq$ such that $\be<q<\ga$. 
\item As subsets of $\bigcup_{S\in\inii}\R^{I_S}_{\lx}$ we have $[\be]_\sme<[\ga]_\sme$.
\end{enumerate}
\end{lemma}

\begin{proof}
Let us prove (1). If $\be,\ga\in\gq$, we may take $q=(\be+\ga)/2$. If $\be,\ga\not\in\gq$, then (1) follows from Lemma \ref{criteri}.

Suppose that $\be=(x_j)_{j\in\I}\not\in\gq$ and $\ga=(y_j)_{j\in\I}\in\gq$. Let $i\in\I$ be the minimal element in $\supp(\ga-\be)$.
Since $x_i<y_i$ in $\R$, there exists $z\in\Q$ such that $x_i<z<y_i$. By Lemma \ref{hahn}, there exists $p\in\gq$ of the form
$$
p=(0\cdots0\mid\,y_i-z\mid\,\star\cdots\star),\qquad p_i=y_i-z.
$$
For $q=\gamma-p\in\gq$, we have $\be<q<\ga$.

The case $\be\in\gq$, $\ga\not\in\gq$ is completely analogous. \e

Let us prove (2). Take $\be'\in[\be]_\sme$, $\ga'\in[\ga]_\sme$. If $\ga<\be'$, then there is some $q\in\gq$ such that $\be<\ga<q<\be'$ by item (1). Since $\be\sim_\sme\be'$, this contradicts either Lemma \ref{bebe} (if $\be$ is commensurable), or Lemma \ref{criteri} (if $\be$ is incommensurable). Hence, we have necessarily $\be'<\ga$. 

Now, if $\ga'<\be'$ then we may find $q\in\gq$ such that $\ga'<\be'<q<\ga$, contradicting Lemma \ref{bebe} or Lemma \ref{criteri}, because  $\ga\sim_\sme\ga'$. Therefore, $\be'<\ga'$. 
\end{proof}

From item (2) of Lemma \ref{density} we deduce that the order on $\gsme$ is canonical. Any other choice of a system of representatives of $\left(\bigcup_{S\in\inii}\R^{I_S}_{\lx}\right)/\!\sim_\sme$ leads to a set which is isomorphic to $\gsme$ as ordered sets.

Also, from item (1) of Lemma \ref{density} we deduce the following result.

\begin{proposition}\label{GQdense}
$\gq$ is dense in $\gsme$. 
\end{proposition}

\subsection{Relative position of the increasing-rank elements}

%Let us emphasize the position of the increasing-rank elements with respect to  the subset $\g_\R$.

The elements associated to the initial segment $S=\emptyset$ are extreme elements:
$$
-\infty=\min\left(\gsme\right),\qquad \infty^-=\max\left(\gsme\right).
$$

For $S=I$, each $b\in\gq$ has an immediate predecessor and an immediate successor:
$$
b^-<b<b^+,\qquad b^-=\max\left(\g_{\op{sme},<b}\right),\quad b^+=\min\left(\g_{\op{sme},>b}\right).
$$

If $\emptyset\subsetneq S \subsetneq I$, then for every commensurable $b\in\R^S_{\lx}$ we have
$$
b^-<\pi_S^{-1}(b)<b^+,\qquad b^-=\max\left(\g_{\op{sme},<\pi_S^{-1}(b)}\right),\quad b^+=\min\left(\g_{\op{sme},>\pi_S^{-1}(b)}\right).
$$

\noindent{\bf Examples.} Let us exhibit the set $\gsme$ in some concrete examples.\e

(0) \ $\g=\{0\}$.\e

In this case, \ $\gq=\grr=\{0\}$ \ and \ $\gsme=\{-\infty,\,0,\,\infty^-\}$. \e

(1) \ $\g=\R$.\e  

In this case, \ $\gq=\grr=\R$ \ and \ $\gsme$ \ is a real line with global minimal and maximal added elements, and such that to every real number an immediate predecessor and successor have been added.

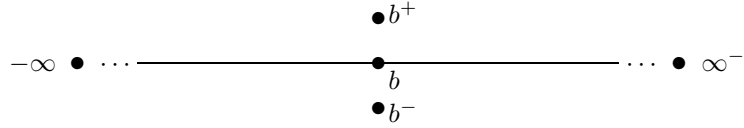
\begin{center}
\setlength{\unitlength}{4mm}
\begin{picture}(16,6)
\put(-2,2){$\bullet$}\put(0.25,2.25){\line(1,0){16}}\put(18,2){$\bullet$}
\put(8,0.5){$\bullet$}\put(8,2){$\bullet$}\put(8,3.5){$\bullet$}
%\multiput(3,.5)(0,.25){22}{\vrule height2pt}
%\multiput(8,.9)(0,.25){9}{\vrule height2pt}
%\multiput(-.1,3)(.25,0){55}{\hbox to 2pt{\hrulefill }}
\put(-4,2){\begin{footnotesize}$-\infty$\end{footnotesize}}
\put(-1,2){\begin{footnotesize}$\cdots$\end{footnotesize}}
\put(16.5,2){\begin{footnotesize}$\cdots$\end{footnotesize}}
\put(19,2){\begin{footnotesize}$\infty^-$\end{footnotesize}}
\put(8.6,1.4){\begin{footnotesize}$b$\end{footnotesize}}
\put(8.6,3.6){\begin{footnotesize}$b^+$\end{footnotesize}}
\put(8.6,.3){\begin{footnotesize}$b^-$\end{footnotesize}}
\end{picture}
\end{center}

(2) \ $\g=\Q$.\e  

In this case, $\gq=\Q$, $\grr=\R$ and $\gsme$ is a real line with global minimal and maximal added elements, and such that to every rational number an immediate predecessor and successor have been added.

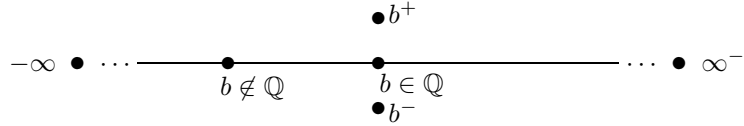
\begin{center}
\setlength{\unitlength}{4mm}
\begin{picture}(16,6)
\put(-2,2){$\bullet$}\put(0.25,2.25){\line(1,0){16}}\put(18,2){$\bullet$}
\put(8,0.5){$\bullet$}\put(8,2){$\bullet$}\put(8,3.5){$\bullet$}
\put(3,2){$\bullet$}
%\multiput(3,.5)(0,.25){22}{\vrule height2pt}
%\multiput(8,.9)(0,.25){9}{\vrule height2pt}
%\multiput(-.1,3)(.25,0){55}{\hbox to 2pt{\hrulefill }}
\put(-4,2){\begin{footnotesize}$-\infty$\end{footnotesize}}
\put(19,2){\begin{footnotesize}$\infty^-$\end{footnotesize}}
\put(-1,2){\begin{footnotesize}$\cdots$\end{footnotesize}}
\put(16.5,2){\begin{footnotesize}$\cdots$\end{footnotesize}}
\put(8.3,1.3){\begin{footnotesize}$b\in\Q$\end{footnotesize}}
\put(8.6,3.6){\begin{footnotesize}$b^+$\end{footnotesize}}
\put(8.6,.3){\begin{footnotesize}$b^-$\end{footnotesize}}
\put(3,1.2){\begin{footnotesize}$b\not\in\Q$\end{footnotesize}}
\end{picture}
\end{center}

(3) \ $\g=\R^2_{\lx}$.\e  

In this case, $\gq=\grr=\R^2_{\lx}$ and $\gsme$ is a real plane with global minimal and maximal added elements, such that every vertical line has a minimal and maximal added element, and to every single point an immediate predecessor and successor have been added.

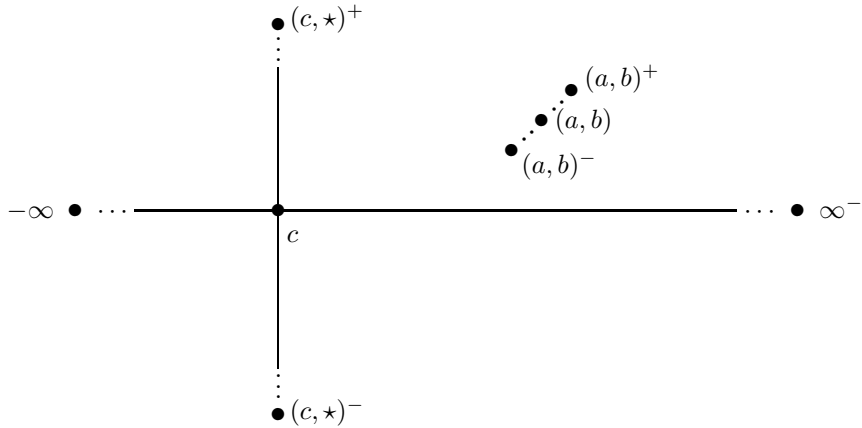
\begin{center}
\setlength{\unitlength}{4mm}
\begin{picture}(22,14)
\put(-4,5){$\bullet$}\put(-1.75,5.25){\line(1,0){20}}\put(20,5){$\bullet$}
\put(11.5,8){$\bullet$}\put(10.5,7){$\bullet$}\put(12.5,9){$\bullet$}
\put(2.75,5){$\bullet$}
\put(3,0){\line(0,1){10}}

\multiput(11.05,7.35)(0.25,0.25){2}{$\cdot$}
\multiput(12.05,8.35)(0.25,0.25){2}{$\cdot$}
\put(12.2,8){\begin{footnotesize}$(a,b)$\end{footnotesize}}
\put(11.1,6.5){\begin{footnotesize}$(a,b)^-$\end{footnotesize}}
\put(13.2,9.4){\begin{footnotesize}$(a,b)^+$\end{footnotesize}}
%\multiput(8,.9)(0,.25){9}{\vrule height2pt}
%\multiput(-.1,3)(.25,0){55}{\hbox to 2pt{\hrulefill }}
\put(-6,5){\begin{footnotesize}$-\infty$\end{footnotesize}}
\put(21,5){\begin{footnotesize}$\infty^-$\end{footnotesize}}
\put(-3,5){\begin{footnotesize}$\cdots$\end{footnotesize}}
\put(18.5,5){\begin{footnotesize}$\cdots$\end{footnotesize}}
\put(3.3,4.2){\begin{footnotesize}$c$\end{footnotesize}}
\put(2.9,-1){\begin{footnotesize}$\vdots$\end{footnotesize}}
\put(2.9,10.2){\begin{footnotesize}$\vdots$\end{footnotesize}}
\put(2.75,11.2){$\bullet$}\put(2.75,-1.8){$\bullet$}
\put(3.4,11.4){\begin{footnotesize}$(c,\star)^+$\end{footnotesize}}
\put(3.4,-1.7){\begin{footnotesize}$(c,\star)^-$\end{footnotesize}}
\end{picture}
\end{center}\bs\bs

(4) \ $\g=\Q^2_{\lx}$.\e  

In this case, $\gq=\g$ and $\grr=\R^2_{\lx}\setminus \{(x,y)\mid x\not\in\Q,\ y\ne0\}$. 

Now, $\gsme$ adds to $\grr$ a global minimal and maximal elements. Also, it adds a minimal and maximal element to each vertical line with rational abscissa. Finally, it adds an immediate predecessor and successor to every single rational point in $\Q^2$.

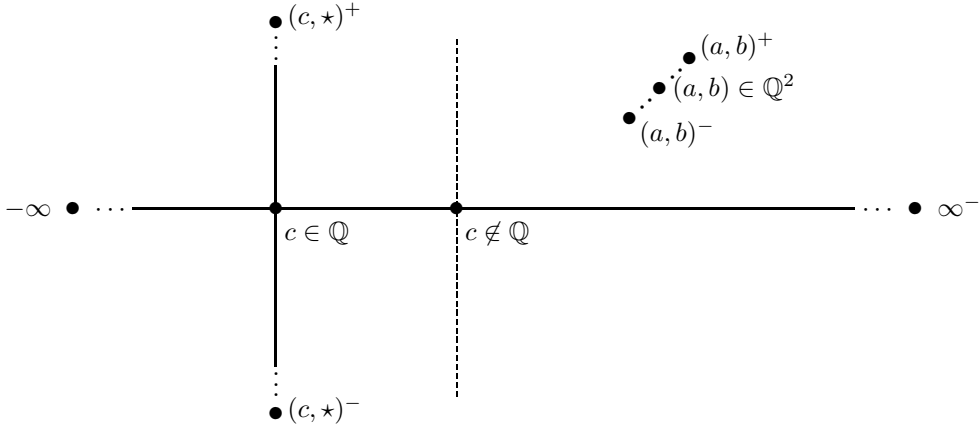
\begin{center}
\setlength{\unitlength}{4mm}
\begin{picture}(28,14)
\put(-4,5){$\bullet$}\put(-1.75,5.25){\line(1,0){24}}\put(24,5){$\bullet$}
\put(15.5,9){$\bullet$}\put(14.5,8){$\bullet$}\put(16.5,10){$\bullet$}
\put(2.75,5){$\bullet$}\put(8.75,5){$\bullet$}
\put(3,0){\line(0,1){10}}

\multiput(15.05,8.35)(0.25,0.25){2}{$\cdot$}
\multiput(16.05,9.35)(0.25,0.25){2}{$\cdot$}
\put(16.2,9){\begin{footnotesize}$(a,b)\in\Q^2$\end{footnotesize}}
\put(15.1,7.5){\begin{footnotesize}$(a,b)^-$\end{footnotesize}}
\put(17.1,10.4){\begin{footnotesize}$(a,b)^+$\end{footnotesize}}
\multiput(9,-1)(0,.3){40}{\vrule height2pt}
%\multiput(-.1,3)(.25,0){55}{\hbox to 2pt{\hrulefill }}
\put(-6,5){\begin{footnotesize}$-\infty$\end{footnotesize}}
\put(25,5){\begin{footnotesize}$\infty^-$\end{footnotesize}}
\put(-3,5){\begin{footnotesize}$\cdots$\end{footnotesize}}
\put(22.5,5){\begin{footnotesize}$\cdots$\end{footnotesize}}
\put(3.3,4.2){\begin{footnotesize}$c\in\Q$\end{footnotesize}}
\put(9.3,4.2){\begin{footnotesize}$c\not\in\Q$\end{footnotesize}}
\put(2.9,-1){\begin{footnotesize}$\vdots$\end{footnotesize}}
\put(2.9,10.2){\begin{footnotesize}$\vdots$\end{footnotesize}}
\put(2.75,11.2){$\bullet$}\put(2.75,-1.8){$\bullet$}
\put(3.4,11.4){\begin{footnotesize}$(c,\star)^+$\end{footnotesize}}
\put(3.4,-1.7){\begin{footnotesize}$(c,\star)^-$\end{footnotesize}}
\end{picture}
\end{center}\bs\bs

%\subsection*{Classification of small extensions}

%As mentioned above, every small extension of $\g$ is equivalent to $\ggb$ for some (non-unique) $\be\in\gsme$.

\subsection{The ordered set $\gsme$ is complete}

%Let $\g$ be an ordered group. Let $I=\pcv(\g)$ be the set of non-zero principal convex subgroups of $\g$, and $$\I=I\cup\left\{i_S\mid S\in\inii\right\}$$ the one-added-element hull introduced in section \ref{secLargerRank}.

In this section, we show that the small-extensions closure $\gsme$ of $\g$ is \emph{complete} as a totally ordered set. 

That is, any non-empty subset $X\subset \gsme$ admits a minimal upper bound $\ga\in\gsme$. 
We say that $\ga$ is the \emph{supremum} of $X$, and we write
$$
\ga=\sup(X)\in\gsme.
$$

The existence of a supremum implies the existence of an \emph{infimum} $$\inf(X)=-\sup(-X)\in\gsme,$$ which is a maximal lower bound of $X$ in $\gsme$. 

In fact, this follows from the following observation.

\begin{lemma}\label{menys}
The mapping $\be\mapsto -\be$ is an order-reversing automorphism of $\gsme$ as an ordered set. 
\end{lemma}

\begin{proof}
The mapping $\be\mapsto -\be$ is an order-reversing automorphism of $\R^\I_{\lx}$ as an ordered group. Hence, we need only to check that
$$
\be\in\gsme\ \imp\ -\be\in\gsme.
$$
This follows immediately from the concrete description of the elements in $\gsme$ given in sections \ref{secSameRank} and \ref{secLargerRank}.
\end{proof}

\noindent{\bf Remark. }Actually, multiplication by $-1$ respects the different strata of $\gsme$:
$$\gsme=\gq\,\sqcup\,\eqrat(\g)\,\sqcup\,\eqri(\g)\,\sqcup\,\incrg.$$

For the proof of the completeness of $\gsme$ we need two auxiliary results. Let us first fix some notation.

For all $S\in\inii$ we denote its one-added-element hull by
$$
\S=S\cup\left\{i_U\mid U\in\inii,\ U\subset S\right\}.
$$

Also, for all $x=(x_i)_{i\in\I}\in\R^\I_{\lx}$ we consider projections
$$
x_S=(x_i)_{i\in S}\in\R^S_{\lx},\qquad  x_\S=(x_i)_{i\in\S}\in\R^\S_{\lx},
$$
and for all $X\subset\gsme$, we denote
$$
X_\S=\left\{x_\S\mid x\in X\right\}.
$$

For $S=\emptyset$, we have $\S=\{i_\emptyset\}$ and $x_\S=x_{i_\emptyset}\in\{-1,0,1\}$. We agree that $\R^\emptyset_{\lx}=\{0\}$.

We shall implicitly assume that $\R^S_{\lx}\subset \R^\S_{\lx}$ under the natural embedding
$$
(x_i)_{i\in S}\ \longmapsto\ (y_j)_{j\in\S},\quad\mbox{where }\ y_j=
\begin{cases}
x_j,&\mbox{ if }j\in S,\\
0,&\mbox{ otherwise}.
\end{cases}
$$
We shall compare elements in $\R^S_{\lx}$ with elements in $\R^\S_{\lx}$ without explicitly recalling this embedding. For instance, for all $x\in\R^\I_{\lx}$ we have
$$
x_S=x_\S \ \sii\ x_i=0\ \mbox{ for all }\ i\in\S\setminus S.
$$

An element $\be\in\R^\S_{\lx}$ is said to be \emph{commensurable} if there exists some $b\in\gq\subset\R^\I_{\lx}$ such that $b_\S=\be$. In this case, we have $b_S=b_\S$.

Finally, let us recall how the incommensurable elements in $\gsme$ were constructed.

\begin{remark}\label{rmkIncom}\rm
If $\rho\in\gsme$ is incommensurable, then we have two possibilities:\e

(a) \ $\rho=(\be\mid0)$ is a minimal incommensurable element in $\rlex$. That is, for some $S\in\inii$, $S\ne\emptyset$, we have $\be\in\R^S_{\lx}$ incommensurable, such that $\be_U$ is commensurable for all $U\in\inii$, $U\subsetneq S$. \e

(b) \ $\rho=(\be\mid\pm1\mid0)$ is a minimal incommensurable element in $\R^\I_{\lx}$, where $\be\in\R^S_{\lx}$ is commensurable for some $S\in\inii$. If $S=\emptyset$, then $\rho=(\pm1\mid0)$.

If $\rho=(\rho_i)_{i\in \I}$, then the unique non-zero coordinate at indices in $\I\setminus I$ is $\rho_{i_S}=\pm1$.
\end{remark}

\begin{lemma}\label{maxcomm}
Let $X\subset\gsme$ be a non-empty subset having no maximal element. Let $S\in\inii$ be an initial segment such that $X_\S$ contains a maximal element, and take $x\in X$ such that $x_\S=\max(X_\S)$. Then:
\begin{enumerate}
\item All $y\in X$, $y\ge x$, have projection $y_\S=x_\S$.
\item For all $U\in\inii$, $U\subset S$, we have $\max(X_\U)=x_\U$.
\item The element $x_\S$ is commensurable. In particular, $x_\S=x_S$ belongs to $\R^S_{\lx}$.
\end{enumerate}
\end{lemma}

\begin{proof}
All $y\in X$, $y\ge x$, satisfy $y_\S\in X_\S$ and $y_\S\ge x_\S$. Since $x_\S=\max(X_\S)$, necessarily $y_\S=x_\S$. This proves (1).

Let $U\in\inii$ such that $U\subset S$. For all $y\in X$ the inequality $y_\S\le x_\S$ implies 
$y_\U\le x_\U$. This proves (2).

Suppose that $x_\S$ is incommensurable. Then, $x$ is incommensurable and Remark \ref{rmkIncom} shows that there exist $U\in\inii$ and $\be\in\R^U_{\lx}$such that:

(a) \ $x=(\be\mid0)\in\rlex$ minimal incommensurable, or

(b) \ $x=(\be\mid\pm1\mid0)\in\R^\I_{\lx}$ and $\be$ is commensurable. In this case, $x_{i_U}=\pm1$. 

Now, if $S\subsetneq U$, then $x_\S=\be_\S$ would be commensurable, against our assumption. Hence, $U\subset S$, and this implies $x=(x_\S\mid0)$. By item (1), for all $y\in X$, $y\ge x$ , the projection  $y_\S=x_\S$ is incommensurable. By Remark \ref{rmkIncom}, we must have $y=(y_\S\mid0)$ too; that is, $y=x$. This implies $\max(X)=x$, against our asumption. 

Therefore, $x_\S$ must be commensurable.  
\end{proof}

\begin{lemma}\label{globalmax}
Let $X\subset\gsme$ be a non-empty subset having no maximal element. Consider the following set of initial segments of $I$:
$$
\mathcal{L}=\left\{S\in\inii\mid X_\S\mbox{ contains a maximal element}\right\}\subset \inii.
$$
Suppose that $T=\bigcup_{S\in\mathcal{L}}S\in\inii$ does not belong to $\mathcal{L}$. Then, there exists a unique $\be=\left(\be_i\right)_{i\in T}\in\R^T_{\lx}$ such that
 $\;\be_S=\max(X_\S)$ \ for all $S\in\mathcal{L}$.

Moreover, this element $\be$ satisfies
$$
x_{\T^0}< \be\quad \mbox{ for all }x\in X,
$$
where $\T^0=\bigcup_{S\in\mathcal{L}}\S=\T\setminus\{i_T\}=\I_{<i_T}$.
\end{lemma}

\begin{proof}
Note that $\emptyset\in\mathcal{L}$ and $I\not\in\mathcal{L}$, by our assumptions on $X$. Thus, $\mathcal{L}$ is a non-empty proper subset of $\inii$.
Since $T\not\in\mathcal{L}$, we have $S\subsetneq T$ for all $S\in\mathcal{L}$.

Denote $\be_S=\max(X_\S)$ for all $S\in\mathcal{L}$. By Lemma \ref{maxcomm}, these elements $\be_S\in\R^S_{\lx}$ are commensurable and form a coherent sequence with respect to inclusions:
$$
U,S\in\mathcal{L},\quad U\subset S\ \imp\ \left(\be_{S}\right)_U=\be_U.
$$
Thus, they determine an element $(\be_i)_{i\in T}\in\R^{T}$ whose coordinates are uniquely determined by:
$$
\be_i=\left(\be_S\right)_i,\quad \mbox{for all }S\in\mathcal{L}\mbox{ such that }i\in S.
$$
%Since all $\be_\S\in\R^\S_{\lx}$ are commensurable, we have $\be_i=0$ for all $i\in \T^0\setminus T$. Thus, we may identify $(\be_i)_{i\in\T^0}$ with some $\be\in\R^T$, through the canonical embedding $\R^T\subset\R^{\T^0}$. 

The property $\be_S=\max(X_\S)$, for all $S\in\mathcal{L}$, follows from the construction of $\be$. 

%If we impose the condition $\be_{i_T}=0$, then  $\be\in\R^\T$ is uniquely determined and satisfies conditions (i) and (ii), by the construction of $\be$. 

Now, let us show that $\be$ belongs to the Hahn product $\R^T_{\lx}$. To this end, we must check that the set $\supp(\be)=\{i\in T\mid \be_i\ne0\}$ is well ordered.

Let $J\subset \supp(\be)$ be a non-empty subset. Take some $j\in J$. There exists $S\in\mathcal{L}$ such that $j\in S$; thus, $J_{\le j}\subset\supp(\be_S)$. Since $\supp(\be_S)$ is well ordered, the set $J_{\le j}$ has a minimal element. This element is the minimal element of $J$  too.  

This proves that $\supp(\be)$ is well ordered, so that $\be\in\R^T_{\lx}$.

Finally, let us  show that $\be>x_{\T^0}$ for all $x\in X$. For all $x\in X$ we have
\begin{equation}\label{maxx}
x_\S\le\be_\S=\be_S\quad\mbox{ for all }S\in\mathcal{L}\ \imp\ x_{\T^0}\le \be_{\T^0}=\be.
\end{equation}

%If for some $S\in\mathcal{L}$ we have $x_\S<\be_\S$, then we deduce $x_{\T^0}<\be$.

Suppose that $x_{\T^0}=\be_{\T^0}=\be$ for some $x\in X$. Let us show that this leads to a contradiction.

%$x_\S=\be_\S=\max(X_\S)$ for all $S\in\mathcal{L}$ leads to a contradiction. Our element $x$ would then satisfy $x_{\T^0}=\be_{\T^0}=\be$. 

For all $y\in X$, $y>x$, we have $y_{\T^0}\ge x_{\T^0}$. By (\ref{maxx}), we deduce that $y_{\T^0}=x_{\T^0}$.

By Remark \ref{rmkIncom}, if $y_{i_T}\ne0$, then we have necessarily $$y=(y_T\mid\pm1\mid0)\quad\mbox{with $y_T$ commensurable and }y_{i_T}=\pm1.$$
Both possibilities lead to a contradicion. If $y=(y_T\mid-1\mid0)$, then $y\le x$, while $y=(y_T\mid1\mid0)$ implies $y=\max(X)$, against our assumptions.

Therefore, all $y\in X$, $y>x$, satisfy $y_{i_T}=0$. Thus, all these elements satisfy $y_\T=\max(X_\T)$, so that $T$ would belong to $\mathcal{L}$, against our assumption.    
\end{proof}

\begin{theorem}\label{completion}
The totally ordered set $\ \gsme$ is complete. 
\end{theorem}

\begin{proof}
Take a non-empty subset $X\subset \gsme$. If $X$ contains a maximal element, then $\sup(X)=\max(X)$. Suppose that $X$ does not contain a maximal element.

Consider the set of initial segments of $I$:
$$
\mathcal{L}=\left\{S\in\inii\mid X_\S\mbox{ contains a maximal element}\right\}\subset \inii.
$$
Let $T=\bigcup_{S\in\mathcal{L}}S\in\inii$. 
We distinguish two cases, requiring different arguments.\e

\noindent{\bf Case 1: \ $T\not\in\mathcal{L}$. }

Consider the element $\be=\left(\be_i\right)_{i\in T}\in\R^T_{\lx}$ constructed in Lemma \ref{globalmax}, satisfying
$\be_S=\max(X_\S)$ for all $S\in\mathcal{L}$, and 
$$x_{\T^0}<\be\ \mbox{ for all }x\in X,$$
where $\T^0=\bigcup_{S\in\mathcal{L}}\S=\T\setminus\{i_T\}=\I_{<i_T}$.

%\begin{equation}\label{xt<}

%\end{equation}

In this case, we claim that
$$
\sup(X)=\ga:=
\begin{cases}
(\be\mid0),&\mbox{ if $\be$ is incommensurable},\\
(\be\mid-1\mid0),&\mbox{ if $\be$ is commensurable}.
\end{cases}
$$

By Remark \ref{rmkIncom}, $\ga\in\gsme$ in both cases, because $\be_S$ is commensurable for all $S\subsetneq T$, by construction.

The inequality $x_{\T^0}<\be=\ga_{\T^0}$ for all $x\in X$, shows that $\ga$ is an upper bound of $X$. Let us check that it is the minimal upper bound in $\gsme$. 

Suppose that $\rho\in\gsme$ satisfies $\rho<\ga$. Then, necessarily $\rho_{\T^0}<\be=\ga_{\T^0}$. 

In fact, suppose that $\rho_{\T^0}=\be$. If $\be$ is commensurable, then $\rho_{i_T}\le \ga_{i_T}=-1$ implies $\rho_{i_T}=-1$, and this implies $\rho=(\be\mid-1\mid0)=\ga$, against our assumption. If $\be$ is incommensurable, then $\rho$ is incommensurable. Since $\rho_i=\be_i=0$ for all $i\in\T^0\setminus T$, necessarily $\rho=(\be\mid0)=\ga$, against our assumption.

Let $i\in \T^0$ be minimal such that $\rho_i<\be_i$. Take $S\in\mathcal{L}$ such that $i\in\S$. There exists $x\in X$ such that $x_\S=\be_\S=\max(X_\S)$. Hence, $\rho<x$.   

Therefore, $\rho$ cannot be an upper bound for $X$.\e

\noindent{\bf Case 2: $T\in\mathcal{L}$. }

In this case, $T=\max(\mathcal{L})$.
Take $x\in X$ such that $x_\T=\max(X_\T)$.
By Lemma \ref{maxcomm}, $x_T=x_\T$ is commensurable.

Since $X$ contains no maximal element, we have necessarily $T\subsetneq I$.  
We distinguish two subcases according to the set $I\setminus T$ having a minimal element or not.\e

\noindent{\bf Case 2a: $I\setminus T$ has no minimal element.}

In this case, we claim that $$\sup(X)=\ga:=(x_T\mid1\mid0)\in\gsme.$$

Indeed, for all $y\in X$, $y\ge x$, Lemma \ref{maxcomm} shows that $y_\T=x_\T$. Thus, $y_T=x_T$ and $y_{i_T}=x_{i_T}=0$, so that $y<\ga$. This shows that $\ga$ is an upper bound for $X$.

Let us show that $\ga$ is the minimal upper bound of $X$ in $\gsme$. Take $\rho=(\rho_i)_{i\in\I}\in\gsme$ such that $x\le\rho<\ga$. Then, $$x_T\le \rho_T\le\ga_T=x_T\ \imp\  x_T=\rho_T.$$ 
Hence, $0=x_{i_T}\le\rho_{i_T}\le\ga_{i_T}=1$, and this implies $\rho_{i_T}=0$.

Indeed, since $\rho\in\gsme$, the equality $\rho_{i_T}=1$ would imply $\rho=(x_T\mid1\mid0)=\ga$, against our assumption. 

Now, since $\supp(x)$ and $\supp(\rho)$ are well-ordered subsets of $\I$, we have:\e

$\bullet$ \ Either $x=(x_T\mid0)$, or there exists $i\in\I_{>i_T}$ minimal satisfying $x_i\ne0$. \e

$\bullet$ \ Either $\rho=(x_T\mid0)$, or there exists $j\in\I_{>i_T}$ minimal satisfying $\rho_j\ne0$.\e 

In any case, since $I\setminus T$ has no minimal element, there exist $k,k'\in I$ such that $T<k<k'$ and 
$$
x_{\ell}=0,\ \rho_\ell=0,\quad \mbox{ for all }\ \ell\in\I,\quad i_T\le\ell<k'.
$$
Hence, for the initial segment $U=I_{<k}$ we have  $x_\U=\rho_\U$. Since $I\setminus T$ has no minimal element, we have $T\subsetneq U$, so that $U$ does not belong to $\mathcal{L}$. Thus, the set 
$X_\U$ does not contain a maximal element. Since the element $x_\U\in X_\U$ cannot be maximal, there exists $y\in X$ such that $y_\U>x_\U=\rho_\U$. Therefore,  $y>\rho$, so that $\rho$ cannot be an upper bound for $X$.\e

\noindent{\bf Case 2b: $I\setminus T$ has a minimal element.}

Let $i=\min(I\setminus T)$. Consider the $i$-th projection of all $y=(y_j)_{j\in\I}\in X$, $y>x$:
$$
Y=\left\{y_i\mid y\in X,\ y>x\right\}\subset \R.
$$

Since the initial segment $S=T\cup\{i\}$ does not belong to $\mathcal{L}$, the set $X_\S$ does not contain a maximal element. By Lemma \ref{maxcomm}, all $y\in X$, $y>x$ have projection
$$
y_\S=(x_\T\mid y_i\mid y_{i_S}).
$$
Since $y_{i_S}$ takes a finite number of values $-1,0,1$, we deduce that the subset $Y\subset \R$ cannot contain a maximal element.
Thus, either $Y$ has no upper bound, or $\sup(Y)=a$ for a real number $a>Y$.

We claim that $\sup(X)=\ga$, for the following $\ga\in\gsme$:
$$
\ga=
\begin{cases}
(x_T,1,0,0\mid0),& \mbox{ if $Y$ not upper bounded},\\
(x_T,0,a,0\mid0),& \mbox{ if }\sup(Y)=a\not\in Q_i,\\
(x_T,0,a,-1\mid0),& \mbox{ if }\sup(Y)=a\in Q_i,\\
\end{cases}
$$
where the three entries right after $x_T$ are the coordinates of $\ga$ at the indices $i_T,i,i_S$, respectively.

%\begin{center}\as{1.3}\begin{tabular}{|c|c|c|c|}\hline$\ga$&$\ga_{i_T}$&$\ga_i$&$\ga_{i_S}$\\\hline$(x_T\mid1\mid0)$&1&0&0\\\hline$(x_T,a\mid0)$&0&$a$&0\\\hline$(x_T,a\mid-1\mid0)$&0&$a$&$-1$\\\hline\end{tabular}\end{center}\e

It is obvious that $\ga$ is an upper bound of $X$. Let us show that it is the minimal upper bound of $X$ in $\gsme$.

Suppose that $\rho=(\rho_j)_{j\in\I}\in\gsme$ satisfies $x\le \rho<\ga$. Since $x_T=\ga_T$, this implies $\rho_T=x_T$. Thus, $0=x_{i_T}\le\rho_{i_T}\le \ga_{i_T}\le1$, and this implies $\rho_{i_T}=0$, because otherwise:
$$
\rho_{i_T}=1 \imp \rho=(x_T\mid1\mid0)\ge\ga,
$$
against our assumption.

If $Y$ is not upper bounded, there exists $y\in X$ such that $y_i>\rho_i$; thus $y>\rho$, and $\rho$ cannot be an upper bound for $X$.

If $\sup(Y)=a\in\R$, then $\rho<\ga$ implies $\rho_i\le a$. We claim that $\rho_i<a$. This inequality implies the existence of  $y\in X$ such that $\rho_i<y_i<a$; thus, $y>\rho$, and $\rho$ cannot be an upper bound for $X$. 

Indeed, suppose that $\rho_i=a$. If $a\not\in Q_i$, then $(x_T,a)$ is incommensurable and since $\rho\in\gsme$, it must be equal to the minimal incommensurable element $(x_T, a\mid0)=\ga$, against our assumption.

If $a\in Q_i$, then  $\rho_S=(x_T, a)=\ga_S$ is commensurable. Since $\rho<\ga$, we deduce that $\rho_{i_S}\le \ga_{i_S}=-1$, and this implies $\rho_{i_S}=-1$. Since $\rho\in\gsme$, it must be equal to  $(x_T,a\mid-1\mid0)=\ga$, against our assumption. 
\end{proof}

\begin{corollary}\label{incomm}
If a non-empty subset $X\subset \gsme$ contains no maximal element, then $\sup(X)$ is incommensurable.
\end{corollary}

\begin{proof}
Suppose that $\sup(X)=\ga\in\gq$. A commensurable element has an immediate predecessor $\ga^-$ in $\gsme$. Since $X$ does not contain $\ga$, this element $\ga^-$ is still an upper bound of $X$, and this contradicts the fact that $\ga$ is the minimal upper bound. 
\end{proof}

Alternatively, this corollary follows from the proof of Theorem \ref{completion}, since in all cases $\sup(X)$ was incommensurable.

\section{Applications to valuation theory}
%\subsection{Valuations on a polynomial ring}

As mentioned in the Introduction, we devote this section to stress the role of $\gsme$ in the description of the valuative tree associated to any valued field $(K,v)$:
$$
\as{1.3}
\begin{array}{ccc}
\ttt_v&\hra &\spv(\kx)\\
\downarrow&&\downarrow\\
\mbox{$[v]$}&\hra &\spv(K)
\end{array}
$$

\subsection{Parameterization of depth-zero paths in $\ttt_v$}\label{subsecPaths}

Let $\g=v(K^*)$ be the value group of $v$.

An element $[\mu]\in\ttt_v$ is an equivalence class of valuations $\mu$ on $\kx$ whose restriction to $K$ are equivalent to $v$. In other words, there exists an embedding of ordered groups $\iota\colon\g\hk\gm$, fitting into a commutative diagram 
\begin{equation}\label{iota}%$$
\as{1.5}
\begin{array}{ccc}
\kx&\stackrel{\mu}\lra&\gm\infty\\
\uparrow&&\ \uparrow\mbox{\tiny$\iota$}\\
K&\stackrel{v}\lra&\g\infty
\end{array}
\end{equation}%$$

%In this case, the valuation induced by $\mu$ on the field $\kappa(\p_\mu)$ is an extension of $v$ to that field.\e

\noindent{\bf Definition.} 
The extension $[\mu]/[v]$ is \emph{commensurable}, \emph{preserves the rank}, or \emph{increases the rank by one}, if the extension $\iota\colon \g\hk\g_\mu$ has this property, respectively. \e

%All valuations with non-trivial support are commensurable over $v$, because the extension $\kappa(\p_\mu)/K$ is algebraic.

Let $I=\pcv(\g)$. By Corollary \ref{embRlx}, we may fix an embedding  of ordered groups, 
$$
\ell\colon \g\hra\rlex\subset\rii. 
$$

\begin{proposition}\label{riiUniverse}
%Let $\mu$ be a valuation on $\kx$ whose restriction to $K$ is equivalent to $v$. Then, there exists a 
Every equivalence class in $\ttt_v$ contains some $\rii$-valued valuation 
$$
\nu\colon \kx\lra \rii\infty.
$$

If $[\nu]/[v]$ is commensurable, preserves the rank, or increases the rank by one, then $\ \gn\subset\gq$, $\ \gn\subset\rlex$, \ or $\ \gn\subset \R^{I_S}_{\lx}$ \ for some $S\in\inii$, respectively.
\end{proposition}

\begin{proof}
Let $\mu$ be an arbitrary valuation on $\kx$ whose restriction to $K$ is equivalent to $v$. There is an embedding $\iota\colon \g\hk\gm$ fitting into a commutative diagram (\ref{iota}). By Theorem \ref{AllSmall}, this embedding $\iota$ is a small extension.

By Lemmas \ref{MaxEqRk} and \ref{SemiUniv}, there is an embedding $\kappa\colon \gm\hk\rii$ fitting into a commutative diagram:   
$$
\as{1.5}
\begin{array}{ccccc}
\kx&\stackrel{\mu}\lra&\gm\infty&&\\
\uparrow&&\ \uparrow\mbox{\tiny$\iota$}&\searrow\!\!\!\raise1.4ex\hbox{\tiny$\kappa$}&\\
K&\stackrel{v}\lra&\g\infty&\stackrel{\ell}\lra&\rii\infty
\end{array}
$$

Take $\nu$ to be the valuation on $\kx$ determined by the mapping $\kappa\circ \mu$. Its value group is $\gn=\kappa(\gm)\subset \rii$, and the tautological isomorphism $\kappa\colon \gm\ism\gn$ shows that the two valuations are equivalent.

The last statement of the proposition is obvious.
\end{proof}

Therefore, in order to describe $\ttt_v$ it is suficient to describe equivalence classs of $\rii$-valued valuations. If $\nu$ is such a valuation, the diagram (\ref{iota}) takes the form:
$$
\as{1.5}
\begin{array}{ccc}
\kx&\stackrel{\nu}\lra&\gn\infty\\
\uparrow&&\ \uparrow\mbox{\tiny$\ell$}\\
K&\stackrel{v}\lra&\g\infty
\end{array}
$$
where $\ell$ is our fixed embedding of $\g$ into $\rii$. 

By the last statement of Proposition \ref{riiUniverse}, $\gn\subset \bigcup_{S\in\inii}\R^{I_S}_{\lx}$.

Let us describe the equivalence classes of the valuations of \emph{depth zero}, in the terminology of \cite{MLV}. \e

%\section{Depth zero valuations on $\kx$}

\noindent{\bf Definition.}
For given $a\in K$ and $\dta\in \left(\bigcup_{S\in\inii}\R^{I_S}_{\lx}\right)\infty$, the depth-zero valuation $\nu=\om_{a,\ga}$ on $\kx$ is defined as 
$$
\nu\left(\sum\nolimits_{0\le s}a_s(x-a)^s\right) = \min\left\{\ell(v(a_s))+s\dta\mid0\le s\right\}.
$$

Note that $\om_{a,\infty}$ has non-trivial support $(x-a)\kx$, because $$\om_{a,\infty}(f)=\ell(v(f(a))),\quad \mbox{for all } \ f\in\kx.$$ On the other hand, for all $\dta\ne\infty$ the valuation $\om_{a,\dta}$ has trivial support. \e

\noindent{\bf Convention. }For simplicity in the exposition, we identify from now on the group $\g$ with its image in $\rii$. In this way, we omit any reference to the embedding $\ell$.\e 
%$$\gm=\begin{cases}\gga&\mbox{ if }\ga<\infty,\\\g&\mbox{ if }\ga=\infty.\end{cases}$$

With this simplification, the value group of $\om_{a,\dta}$ is: 
$$\g_{\om_{a,\dta}}=\begin{cases}
\g,&\quad\mbox{ if }\ \dta=\infty,\\
\gen{\g,\dta},&\quad\mbox{ otherwise}.
\end{cases}
$$

\begin{lemma}\label{ClassDth0}
Take $a\in K$, and $\dta,\ep\in\bigcup_{S\in\inii}\R^{I_S}_{\lx}$. The depth-zero valuations $\om_{a,\dta}$ and $\om_{a,\ep}$ are equivalent if and only if $\dta\sim_\sme\ep$.
\end{lemma}

\begin{proof}
We have  $\om_{a,\dta}\sim\om_{a,\ep}$ if and only if there exists an order-preserving isomorphism $$\varphi\colon \ggd=\g_{\om_{a,\dta}}\iso \g_{\om_{a,\ep}}=\gge,$$ fitting into a commutative diagram
$$
\as{1.4}
\begin{array}{ccc}
\ggd\infty&\stackrel{\varphi}\lra\ &\!\!\gge\infty\\
\quad\ \mbox{\scriptsize$\om_{a,\dta}$}&\nwarrow\ \nearrow&\!\!\!\mbox{\scriptsize$\om_{a,\ep}$}\quad\\
&\kx&
\end{array}
$$
This is equivalent to $\varphi$ acting as the identity on $\g$ and mapping $\dta$ to $\ep$. In other words, it is equivalent to $\dta\sim_\sme\ep$.
\end{proof}

Therefore, for any fixed $a \in K$, the set $\gsme\infty$ parameterizes a certain path in $\ttt_v$, represented by the valuations $\om_{a,\dta}$ with $\dta\in\gsme\infty$:

\begin{center}
\setlength{\unitlength}{4mm}
\begin{picture}(16,4)
\put(-2,1){$\bullet$}\put(0.25,1.3){\line(1,0){16}}\put(18,1){$\bullet$}
\put(6,1){$\bullet$}\put(20,1){$\bullet$}
%\multiput(3,.5)(0,.25){22}{\vrule height2pt}
%\multiput(8,.9)(0,.25){9}{\vrule height2pt}
%\multiput(-.1,3)(.25,0){55}{\hbox to 2pt{\hrulefill }}
\put(-3,2){\begin{footnotesize}$\om_{a,-\infty}$\end{footnotesize}}
\put(-1,1){\begin{footnotesize}$\cdots$\end{footnotesize}}
\put(16.5,1){\begin{footnotesize}$\cdots$\end{footnotesize}}
\put(17,2){\begin{footnotesize}$\om_{a,\infty^-}$\end{footnotesize}}
\put(21,1){\begin{footnotesize}$\om_{a,\infty}$\end{footnotesize}}
\put(6,2){\begin{footnotesize}$\om_{a,\dta}$\end{footnotesize}}
\end{picture}
\end{center}\e

Moreover, $[\om_{a,\dta}]/[v]$ is commensurable if and only if $\dta\in\gq$. Also, $[\om_{a,\dta}]/[v]$ preserves the rank if and only if $\dta\in\grr$.

The minimal class  $[\om_{a,-\infty}]$ is represented by the valuation $\minf:=\om_{a,-\infty}$, which does not depend on $a$:
$$
\minf\colon\kx\longtwoheadrightarrow \left(\Z\times\g\right)\infty,\qquad f\longmapsto \left(\ord_\infty(f),v(\lc(f))\right), 
$$
where $\lc(f)$ is the leading coefficient of a non-zero polynomial $f$.

Actually, we shall show elsewhere that $\tt_v$ is an oriented connected tree admiting the equivalence class of $\minf$ as the unique root node.

The maximal class $[\om_{a,\infty^-}]$ is represented by the valuation
$$
\om_{a,\infty^-}\colon\kx\longtwoheadrightarrow \left(\Z\times\g\right)\infty,\qquad f\longmapsto \left(\ord_{x-a}(f),v(\init(f))\right), 
$$
where $\init(f)$ is the first  non-zero coefficient of the $(x-a)$-expansion of $f\in\kx$.
\e

What is the relative position of the paths corresponding to two different elements $a,b\in K$? 

\begin{lemma}\label{classif}
Let $a,b\in K$ and $\dta,\ep\in  \bigcup_{S\in\inii}\R^{I_S}_{\lx}$.

\begin{enumerate}
\item $\ \om_{a,\dta}=\om_{b,\ep}\ \sii\ v(b-a)\ge\dta=\ep$.
\item $\ \om_{a,\dta}\sim\om_{b,\ep}\ \sii\ v(b-a)\ge[\dta]_\sme=[\ep]_\sme$.
\end{enumerate}
\end{lemma}

\begin{proof}
Item (1) is well-known. Let us prove (2).

Suppose that $\om_{a,\dta}\sim\om_{b,\ep}$. Then, there exists an order-preserving isomorphism $\varphi\colon \ggd\iso \gge$, fitting into a commutative diagram
$$
\as{1.4}
\begin{array}{ccc}
\ggd\infty&\stackrel{\varphi}\lra\ &\!\!\gge\infty\\
\quad\ \mbox{\scriptsize$\om_{a,\dta}$}&\nwarrow\ \nearrow&\!\!\!\mbox{\scriptsize$\om_{b,\ep}$}\quad\\
&\kx&
\end{array}
$$
Since $\left(\om_{a,\dta}\right)_{\mid K}=\left(\om_{b,\ep}\right)_{\mid K}$, the isomorphism $\varphi$ acts as the identity on $\g$. On the other hand, 
$$\varphi(\dta)=\om_{b,\ep}(x-a)= \min\{\ep,v(b-a)\}\le\ep.
$$
$$\varphi^{-1}(\ep)=\om_{a,\dta}(x-b)= \min\{\dta,v(b-a)\}\le\dta.
$$
Hence, $\varphi(\dta)=\ep$, and this implies $[\dta]_\sme=[\ep]_\sme$ by the definition of $\sim_\sme$.

A posteriori, we deduce from $\varphi(\dta)=\ep$ that $v(b-a)\ge\dta$. If $\dta\in \gq$,  then  Lemma \ref{bebe} shows that $v(b-a)\ge\{\dta\}=[\dta]_\sme$.

If $\dta\not\in\gq$, then $v(b-a)>\dta$ and this implies $v(b-a)>[\dta]_\sme$ by Lemma \ref{density},(2).\e

Conversely, suppose that $v(b-a)\ge[\dta]_\sme=[\ep]_\sme$. Let $\be\in\gsme$ the unique element in $\gsme$ lying in the class $[\dta]_\sme=[\ep]_\sme$. By Lemma \ref{ClassDth0}, 
$$
\om_{a,\dta}\sim \om_{a,\be},\qquad \om_{b,\ep}\sim \om_{b,\be}.
$$
Now, the condition $v(b-a)\ge\be$ implies $\om_{a,\be}=\om_{b,\be}$ by item (1). Thus, $\om_{a,\dta}\sim \om_{b,\ep}$. 
\end{proof}

\begin{corollary}\label{classif0}
 Let $a,b\in K$ and let $\dta,\ep\in\gsme$. We have
 $$
 \om_{a,\dta}\sim \om_{b,\ep}  \ \sii\ 
 \om_{a,\dta}=\om_{b,\ep}  \ \sii\ v(b-a)\ge \dta=\ep.
 $$
\end{corollary}

%By Lemma \ref{classif},(1), the minimal valuation $\om_{a,-\infty}$ is independent of $a$:$$\om_{a,-\infty}=\om_{b,-\infty}\quad \mbox{for all }a,b\in K.$$We shall denote this absolute minimal depth-zero valuation simply by $\minf$. By thinking in the case $a=0$, we see that it acts as follows:

Therefore,  the depth-zero paths in $\ttt_v$ determined by any two $a,b\in K$ coincide for all parameters $\dta\in\gsme$ in the interval $[-\infty,v(a-b)]\subset\gsme$. 

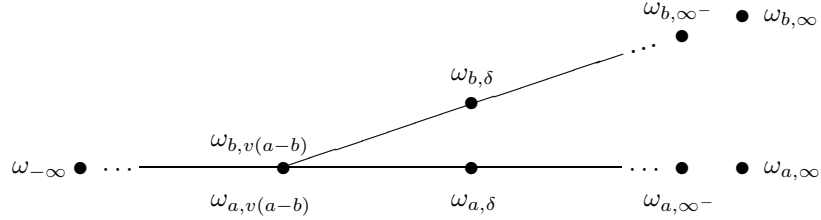
\begin{center}
\setlength{\unitlength}{4mm}
\begin{picture}(22,8)
\put(-2,1){$\bullet$}\put(4.75,1){$\bullet$}
\put(11,1){$\bullet$}\put(11,3.16){$\bullet$}

\put(18,1){$\bullet$}\put(20,1){$\bullet$}

\put(18,5.4){$\bullet$}\put(20,6.06){$\bullet$}

\put(0.25,1.3){\line(1,0){16}}\put(5,1.3){\line(3,1){11.3}}
\put(16.66,4.48){$\dot{}$}\put(16.99,4.59){$\dot{}$}\put(17.32,4.7){$\dot{}$}
%\multiput(3,.5)(0,.25){22}{\vrule height2pt}
%\multiput(8,.9)(0,.25){9}{\vrule height2pt}
%\multiput(-.1,3)(.25,0){55}{\hbox to 2pt{\hrulefill }}
\put(-4,1.1){\begin{footnotesize}$\minf$\end{footnotesize}}
\put(-1,1){\begin{footnotesize}$\cdots$\end{footnotesize}}
\put(16.5,1){\begin{footnotesize}$\cdots$\end{footnotesize}}
\put(17,0){\begin{footnotesize}$\om_{a,\infty^-}$\end{footnotesize}}
\put(17,6.4){\begin{footnotesize}$\om_{b,\infty^-}$\end{footnotesize}}
\put(21,1.1){\begin{footnotesize}$\om_{a,\infty}$\end{footnotesize}}
\put(21,6.16){\begin{footnotesize}$\om_{b,\infty}$\end{footnotesize}}
\put(2.6,2){\begin{footnotesize}$\om_{b,v(a-b)}$\end{footnotesize}}
\put(2.6,0){\begin{footnotesize}$\om_{a,v(a-b)}$\end{footnotesize}}
\put(10.6,0){\begin{footnotesize}$\om_{a,\dta}$\end{footnotesize}}
\put(10.6,4.2){\begin{footnotesize}$\om_{b,\dta}$\end{footnotesize}}
\end{picture}
\end{center}\e

\subsection{Topological interpretation of valuations on $\kx$}\label{subsecAnalysis}

Let $\mu\colon\kx\to\La\infty$ be a valuation on $\kx$ with trivial support, whose restriction to $K$ is equivalent to $v$. 

Consider the well-known inequality \cite[Thm. 3.4.3]{valuedfield}:
\begin{equation}\label{ineq}
\rrk(\gm/\g)+\op{tr.deg}(\km/k_v)\le 1.
\end{equation}

The valuations for which equality holds in  (\ref{ineq}) are called \emph{valuation-transcendental} in \cite{Kuhl}. Equivalently, they are \emph{finite-depth} valuations in the terminology of \cite{MLV}, or \emph{``bien specifi\'ee"} valuations in the terminology of \cite{Vaq2}.

These valuation-transcendental valuations may be characterized too, as the restrictions to $\kx$ of depth-zero valuations on $\kb[x]$ \cite[Sec. 3]{Kuhl}, \cite[Sec. 3]{Vaq2}.

On the other hand,  depth-zero valuations on $\kb[x]$ have a natural topological interpretation. 

Let us fix  an extension  $\vb$ of $v$ to $\kb$. The value group of $\vb$ is $\g_{\vb}=\gq$ and the residue field $k_{\vb}$ is an algebraic closure of $k_v$.

For $a\in\kb$ and $\dta\in\La$, consider the ultrametric ball with center $a$ and radius $\dta$:
$$
B(a,\dta)=\{c\in \kb\mid \vb(c-a)\ge \dta\}\subset \kb.
$$
Two ultrametric balls of the same radius, are either disjoint or they coincide.

The pair $(a,\dta)$ determines a depth-zero valuation on $\kb[x]$ too:
$$
\om_{a,\dta}\left(\sum\nolimits_{0\le s}a_s(x-a)^s\right) = \min\left\{\vb(a_s)+s\dta\mid0\le s\right\}.
$$
By Lemma \ref{classif}, $\ \om_{a,\dta}=\om_{b,\ep}$ if and only if $B(a,\dta)=B(b,\ep)$. Thus, depth-zero valuations on $\kb[x]$ are objects intrinsically associated to ultrametric balls in $\kb$.

By the results of section \ref{subsecPaths}, every depth-zero valuation on $\kb[x]$ is equivalent to $\om_{a,\dta}$, for some $\dta\in\gsme$.

Our aim in this section, is the following characterization of the valuation-transcen\-den\-tal valuations on $\kx$.

\begin{proposition}
Let $\mb=\om_{a,\dta}$ for some $a\in \kb$ and $\dta\in\gsme$. Let $\mu$ be the valuation on $\kx$ obtained by restriction of $\mb$. Then,
for all $f\in\kx$, we have:

\begin{enumerate}
\item If $\dta\in\gq$, then  $\ \mu(f)=\min\{\vb(f(b))\mid b\in B_{a,\dta}\}$.
\item If $\dta\not\in\gq$, then $\  \mu(f)\sim_\sme\,\inf\{\vb(f(b))\mid b\in B_{a,\dta}\}$.
\end{enumerate}
\end{proposition}

\begin{proof}
Let $f=\sum_{0\le s}a_s(x-a)^s$ be the $(x-a)$-expansion of some non-zero $f\in\kx$.

Let us denote $B=B(a,\dta)$. For all $b\in B$, we have
\begin{equation}\label{generic}
\vb(f(b))\ge \min\{\vb(a_s(b-a)^s)\mid 0\le s\}\ge \min\{\vb(a_s)+s\dta\mid 0\le s\}=\mu(f).
\end{equation}

Let $S$ be the set of indices $s$ such that $\mu(f)=\vb(a_s)+s\dta$.

If $0\in S$, then $\mu(f)=\vb(a_0)=\vb(f(a))$. Since $a\in B$, both statements of the proposition hold, in this case.

From now on, we suppose $0\not\in S$. Let us first discuss the case $\dta\in\gq$.

In this case, $\ga:=\mu(f)$ belongs to $\gq$ too. Let us take $z,u,c\in\kb$ such that:
$$
\vb(z)=0,\qquad \vb(u)=\dta,\qquad \vb(c)=\ga. 
$$

Take $b=a+zu$. Since $\vb(b-a)=\vb(zu)=\dta$, we have $b\in B$. Item (1) will be proved, if we show that $\vb(f(b))=\ga$.

Consider the polynomial:
$$
f_0=\sum_{s\in S}a_s(x-a)^s.
$$
For all $s\not\in S$ we have $\vb(a_s(b-a)^s)=\vb(a_s)+s\dta>\ga$. Hence, 
$$
\vb(f(b))=\ga \ \sii\ \vb(f_0(b))=\ga.
$$

For all $s\in S$, we have $\vb(a_su^s)=\vb(a_s)+s\dta=\ga$. Hence, $\vb(a_su^sc^{-1})=0$. Consider the class of this element modulo the maximal ideal of $\vb$:
$$
\zeta_s:=\overline{a_su^sc^{-1}}\in k_{\vb}.
$$

By (\ref{generic}), $\vb(f_0(b))\ge \ga$. Hence,
\begin{align*}
\vb(f_0(b))>\ga\ &\;\sii\ \vb\left(\sum_{s\in S}a_su^sz^s \right)>\ga\ \sii\ \vb\left(\sum_{s\in S}(a_su^sc^{-1})z^s \right)>0\\
&\; \sii \ \sum_{s\in S}\zeta_s \,\overline{z}^s=0.
\end{align*}

Therefore, any choice of $z\in\kb^*$ such that $\vb(z)=0$ and $\overline{z}\in k_{\vb}^*$ is not a root of the polynomial $\sum_{s\in S}\zeta_sx^s\in k_{\vb}[x]$, leads to some $b\in B$ for which $\vb(f_0(b))=\ga$.

This ends the proof of (1).\e

Suppose now $\dta\not\in\gq$. In particular, $\vb(b-a)>\dta$ for all $b\in B$.

In this case, $S=\{s_0\}$ is a one-element set, because for all $s\ne t$ we have
$$
\vb(a_s)+s\dta=\vb(a_t)+t\dta\ \imp \dta=(\vb(a_s)-\vb(a_t))/(t-s)\in\gq.
$$
Also, $\ga=\mu(f)=\vb(a_{s_0})+s_0\dta$ does not belong to $\gq$. 

Since the case $0\in S$ has been analyzed before, we may assume $s_0>0$.
Consider
$$
\ep:=\min\left\{\dfrac{\vb(a_s)-\vb(a_{s_0})}{s_0-s}\ \Big|\ 0\le s<s_0\right\}\in\gq.
$$

We have $\ep>\dta$, because for all $0\le s<s_0$: 
$$
\vb(a_s)+s\dta>\vb(a_{s_0})+s_0\dta.
$$

Item (2) follows immediately from the following Claim.\e

\noindent{\bf Claim. }If $\be\in\gsme$ satisfies $\ga<\be$ and $[\be]_\sme\ne [\ga]_\sme$, there exists $b\in B$ such that 
$$
\ga<\vb(f(b))<\be.
$$

Let us prove the Claim. We have $$\ga=\vb(a_{s_0})+s_0\dta<\vb(a_{s_0})+s_0\ep$$ and $[\ga]_\sme\ne[\vb(a_{s_0})+s_0\ep]_\sme$, because $\ga\not\in\gq$ but 
$\vb(a_{s_0})+s_0\ep\in\gq$. Hence,
$$
\ga<\min\{\be,\vb(a_{s_0})+s_0\ep\}\quad\mbox{ and }\quad [\ga]_\sme\ne [\min\{\be,\vb(a_{s_0})+s_0\ep\}]_\sme. 
$$
By Lemma \ref{density}, there exists $q\in\gq$ such that 
$$
\ga<q<\min\{\be,\vb(a_{s_0})+s_0\ep\}.
$$
Let us write $q=\vb(a_{s_0})+s_0\rho$ for an adequate $\rho\in\gq$. In this way, we have:
$$
\dta<\rho<\ep.
$$

Take any $b\in\kb$ such that $\vb(b-a)=\rho$. Since $\rho>\dta$, we have $b\in B$.

The Claim will be proven if we show that $$\vb(f(b))=\vb(a_{s_0}(b-a)^{s_0})=\vb(a_{s_0})+s_0\rho=q.$$ 

Since, $f(b)=\sum_{0\le s}a_s(b-a)^s$, it suffices to show that
$$\vb(a_s(b-a)^s)>\vb(a_{s_0}(b-a)^{s_0}),\quad \mbox{ for all }s\ne s_0.
$$
For $s<s_0$ this inequality follows from $\rho<\ep$:
$$
\dfrac{\vb(a_s)-\vb(a_{s_0})}{s_0-s}\ge\ep>\rho\ \imp\ \vb(a_s)+s\rho>\vb(a_{s_0})+s_0\rho.
$$
For $s>s_0$ it follows directly from $\vb(b-a)>\dta$:
$$
\vb(a_s)+s\dta>\vb(a_{s_0})+s_0\dta\ \imp\ \vb(a_s)+s\vb(b-a)>\vb(a_{s_0})+s_0\vb(b-a).
$$
This ends the proof of the Claim.
\end{proof}

\end{document}